\documentclass[10pt]{amsart}
\usepackage[utf8x]{inputenc}
\usepackage{amsthm,amssymb, amsmath}
\usepackage{marvosym}
\usepackage[all]{xy}
\usepackage[pdftex,
            pdfauthor={Your Name},
            pdftitle={The Title},
            pdfsubject={The Subject},
            pdfkeywords={Some Keywords},
            pdfproducer={Latex with hyperref, or other system},
            pdfcreator={pdflatex, or other tool}]{hyperref}
\usepackage{graphicx}
\usepackage[margin=5pt,font=small]{caption}
\usepackage{tikz}
\usetikzlibrary{trees}

\newtheorem{theorem}{Theorem}[section]
\newtheorem{corollary}[theorem]{Corollary}

\newtheorem{lemma}[theorem]{Lemma}
\newtheorem{proposition}[theorem]{Proposition}

\newtheorem{definition}[theorem]{Definition}

\newcommand{\matr}[4]{
\left( \begin{array}{cc} #1 & #2 \\ #3 & #4 \end{array} \right)}
\newcommand{\vect}[2]{
\left( \begin{array}{c} #1 \\ #2 \end{array} \right)}
\newcommand{\fect}[2]{
\left[ \begin{array}{c} #1 \\ #2 \end{array} \right]}

\newcommand{\RR}{\mathbb{R}}
\newcommand{\QQ}{\mathbb{Q}}
\newcommand{\ZZ}{\mathbb{Z}}
\newcommand{\NN}{\mathbb{N}}

\newcommand{\T}{{\bf T}}

\newcommand{\EE}{\mathcal{E}}
\newcommand{\EKU}{\mathcal{E}_{KU}}
\newcommand{\EB}{\mathcal{E}_{B}}

\newcommand{\ku}{qumterval\ }

\newcommand{\cfe}{continued fraction expansion\ }

\DeclareMathOperator*{\var}{Var}

\title{Continued fractions with $SL(2, \mathbb{Z})$-branches: combinatorics and entropy} 

\author{Carlo Carminati, Stefano Isola, Giulio Tiozzo}
\address[Carlo Carminati]{Dipartimento di Matematica \\
Universit\`a di Pisa \\ Largo Bruno Pontecorvo 5, 56127 Pisa, Italy}
\email[Carlo Carminati]{carminat@dm.unipi.it}
\address[Stefano Isola]{Dipartimento di Matematica e Informatica\\
Università di Camerino\\
via Madonna delle Carceri, 62032 Camerino, Italy}
\email[Stefano Isola]{stefano.isola@unicam.it}
\address[Giulio Tiozzo]{ICERM \\
Brown University \\
121 South Main St, Providence RI 02903, USA}
\email[Giulio Tiozzo]{Giulio\_Tiozzo@brown.edu}

\begin{document}

\begin{abstract}

We study the dynamics of a family $K_\alpha$ of discontinuous interval
maps whose (infinitely many) branches are M\"obius transformations in $SL(2, \mathbb{Z})$, 
and which arise as the critical-line case of the family of $(a, b)$\emph{-continued fractions}.

We provide an explicit construction of the bifurcation locus $\EKU$
for this family, showing it is parametrized by Farey words and it has Hausdorff dimension zero.
As a consequence, we prove that the metric entropy of $K_\alpha$ 
is analytic outside the bifurcation set but not differentiable at points of $\EKU$, 
and that the entropy is monotone as a function of the parameter.

Finally, we prove that  the bifurcation set is combinatorially isomorphic
to the main cardioid in the Mandelbrot set, providing one more entry to 
the dictionary developed by the authors between continued fractions and complex dynamics.







\end{abstract}

\maketitle
 
\section{Introduction}


It is well-known that the usual continued fraction algorithm is encoded by the dynamics of the Gauss map
$G(x) := \frac{1}{x} - \lfloor \frac{1}{x} \rfloor$; 
moreover, the Gauss map is known to be related, via a Poincar\'e section, to the geodesic flow on the 
modular surface $\mathbb{H}/SL(2,\mathbb{Z})$. 
In greater generality, the modular group $SL(2, \mathbb{Z})$ is generated by the 
transformations ${\bf S}x := -1/x$ and ${\bf T}x := x + 1$, and several different 
continued fraction algorithms have been constructed by applying the generators according 
to different rules.

In particular, for each $\alpha$ we can construct an interval map $K_\alpha$ by fixing a ``fundamental interval'' 
$[\alpha-1, \alpha)$, and at each step applying the inversion ${\bf  S}$ followed by as many translations ${\bf T}$ 
as are needed to come back 
to the fundamental domain. 
Thus, for each $\alpha\in(0,1)$, we have the interval map 
$K_\alpha :[\alpha -1, \alpha] \to [\alpha -1, \alpha]$ defined by $K_\alpha (0)=0$ and 
$$ K_\alpha (x)= -\frac{1}{x} - c_\alpha(x)$$
where $c_\alpha(x) \in \mathbb{Z}$ is chosen so that the result lies in $[\alpha-1, \alpha)$.
Each $K_\alpha$ determines a continued fraction expansion of type 
$$x = -\frac{1}{c_1 - \frac{1}{c_2 - \frac{1}{c_3 -\dots}}}$$
with coefficients $c_n := c_\alpha(K_\alpha^{n-1}(x))$.
Similarly to the Gauss map, each $K_\alpha$ has infinitely many expanding branches and 
a unique absolutely continuous invariant measure $\mu_\alpha$.

In recent years, S. Katok and I. Ugarcovici, following a suggestion of D. Zagier, 
defined the two-dimensional family $f_{a,b}$ of \emph{$(a, b)$-continued fraction transformations}
and studied their dynamics and natural extensions \cite{KU1, KU2}. 
The maps $K_\alpha$ are the first return maps of $f_{\alpha-1, \alpha}$ on the interval 
$[\alpha-1, \alpha)$ and, as it will be explained, they capture all the essential dynamical features.

The family $K_\alpha$ interpolates between other well-known continued fraction 
algorithms: in particular, for $\alpha = 1/2$ one gets the \emph{continued fraction 
to the nearest integer} going back to Hurwitz \cite{Hu}, while for $\alpha = 1$ one gets the \emph{backward continued fraction}
which is related to the reduction theory of quadratic forms \cite{Zagier, Kat}. 

The definition of $K_\alpha$ is very similar to the definition of the \emph{$\alpha$-continued fraction 
transformations} $T_\alpha$ introduced by Nakada \cite{N} and subsequently studied by several authors
\cite{LuMa, Kra, NN, CT, BCIT, KSS, CT3, AS}. 
In this paper we shall use techniques similar to the ones in \cite{CT} to study the $K_\alpha$: 
as we shall see in greater detail, this will also highlight the substantial differences
in the combinatoral structures of the respective bifurcation sets.
In particular, we shall see that the bifurcation set is canonically isomorphic to the set of external rays landing 
on the main cardioid of the Mandelbrot set (while the bifurcation set for the $\alpha$-continued fractions 
was shown to be isomorphic to the real slice of the Mandelbrot set \cite{BCIT}). 

From a dynamical systems perspective, we shall be interested in studying the variation of the dynamics of 
$K_\alpha$ as a function of the parameter. As we shall see, there exist infinitely many islands of ``stability'', 
and each of them  corresponds to a \emph{Farey word} (see section \ref{S:farey}). 
Namely, to each Farey word $w$ we shall associate an open interval $J_w \subseteq [0, 1]$
called \emph{quadratic maximal interval}, or \emph{qumterval} for short (see section \ref{S:qumtervals}); the \emph{bifurcation set} 
$\EKU$ is defined as the complement of all such intervals: 
$$\EKU := [0, 1] \setminus \bigcup_{w \in FW} J_w.$$
The set $\EKU$ is homeomorphic to a Cantor set and has Hausdorff dimension zero (Proposition \ref{Hdim}).
We shall prove that on each $J_w$ we have the following matching between the orbits of $\alpha$ and $\alpha-1$; namely, 
there exist integers $m_0$ and $m_1$ (which depend only on $J_w$) such that 
\begin{equation}\label{eq:matching0}
K_\alpha^{m_0+1}(\alpha-1)=K_\alpha^{m_1+1}(\alpha)
\end{equation}
for all $\alpha \in J_w$. 

One way to study the bifurcations of the family $K_\alpha$ is by considering its entropy, in the spirit of 
\cite{LuMa}.
Indeed, let us define $h(\alpha)$ to be the metric entropy of the map $K_\alpha$ with respect to the 
measure $\mu_\alpha$. 
We shall prove that the set $\EKU$ is 
precisely the set of 
parameters for which the entropy function is not smooth:

\begin{theorem}\label{T:more}
The function $\alpha \mapsto h(\alpha)$ 
\begin{enumerate}
\item is analytic on $[0,1]\setminus \EKU$;
\item is not differentiable (and not locally Lipschitz) at any $\alpha \in \EKU$.
\end{enumerate}
\end{theorem}
Thus, as the parameter $\alpha$ varies, the dynamics of $K_\alpha$ goes through 
infinitely many stable regimes, one for each connected component of the complement of $\EKU$.
We shall prove, however, that the entropy function is globally monotone across the bifurcations.
In order to state the theorem, let us note that the graph of the entropy function is symmetric 
with respect to the transformation $\alpha \mapsto 1-\alpha$, because $K_\alpha$ and $K_{1-\alpha}$ are measurably conjugate 
(see equation \eqref{eq:ssimmetria}). Moreover, it is not hard by an explicit computation  
to see that the entropy is constant (equal to $\frac{\pi^2}{6 \log (1+g)}$) on the interval $[g^2, g]$, where 
$g := \frac{\sqrt{5}-1}{2}$ is the golden mean (so $g^2  = 1-g = \frac{3-\sqrt{5}}{2}$).

The main theorem is the following monotonicity result for $h$:

\begin{theorem}\label{T:main}
The function $\alpha \mapsto h(\alpha)$ is strictly monotone increasing on $[0,g^2]$, constant on $[g^2,g]$ and 
strictly monotone decreasing on $[g,1]$.
\end{theorem}





Note that Theorem \ref{T:main} highlights a major difference with the $\alpha$-continued fraction case, where the 
entropy is not monotone \cite{NN} in any neighbourhood of $\alpha = 0$, and actually 
the set of parameters where the entropy is locally non-monotone has Hausdorff dimension $1$ \cite{CT3}.
For the $K_\alpha$, the study of the metric entropy was introduced by Katok and Ugarcovici in \cite{KU1}, \cite{KU2}, 
who gave an algorithm to produce the natural extension for any given element in the complement of $\EKU$; 
as a consequence, they computed the entropy in some particular cases. The present work gives a global approach 
which makes it possible to study the entropy as a function of the parameter. 

Condition \eqref{eq:matching0} was introduced in \cite{KU1}, where it is called {\em cycle property}, 
and it is also completely analogous to the {\em  matching
  condition} used by Nakada and Natsui \cite{NN} to study the family $(T_\alpha)$. 

Finally, we shall prove (Proposition \ref{P:limit}) that the entropy tends to $0$ 
as $\alpha \to 0^+$, and there its modulus of continuity is of order $\frac{1}{|\log \alpha|}$
(which is the same behaviour as in the case of $\alpha$-continued fractions).



\begin{figure}
\includegraphics[scale=0.5]{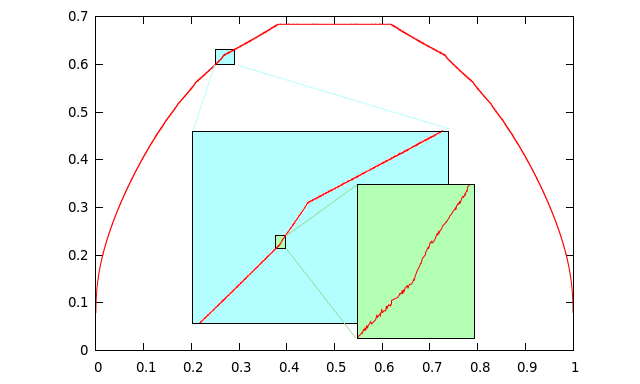} 
\caption{The entropy of $K_\alpha$ as a function of $\alpha$, and a sequence of zooms around a parameter
in the bifurcation set $\EKU$. Note the slope is increasing in each zoom, due to the fact that the entropy is not locally Lipschitz 
at points of $\EKU$ (Theorem \ref{T:more}). However, the entropy is globally monotone on $[0, \frac{3-\sqrt{5}}{2}]$, 
as stated in Theorem \ref{T:main}.}
\label{fig:zoom} 
\end{figure}





\subsection{Connection with the main cardioid in the Mandelbrot set}

The fact that each connected component of the complement of $\EKU$ is naturally labelled by a Farey word
can be used to draw an unexpected connection between the combinatorial structure of $\EKU$ and the 
Mandelbrot set. 

Recall the \emph{main cardioid} of the Mandelbrot set is the set of parameters $c \in \mathbb{C}$
for which the map $f_c(z) := z^2+c$ has an attractive or indifferent fixed point.
The exterior of the Mandelbrot set admits a canonical uniformization map, and to each angle 
$\theta \in \mathbb{R}/\mathbb{Z}$ there corresponds an associated \emph{external ray} $R(\theta)$. Let us denote 
$\Omega$ to be the set of angles $\theta$ for which the ray $R(\theta)$ lands on the main cardioid.

Recall Minkowski's \emph{question mark function} $Q : [0, 1] \to [0, 1]$ is a homeomorphism of the 
interval which is defined by converting the continued fraction 
expansion of a number into a binary expansion. More precisely, if $x = [0; a_1, a_2, a_3, \dots]$ is 
the usual continued fraction expansion of $x$, then we define 
$$Q(x) := 0.\underbrace{0\dots0}_{a_1 - 1}\underbrace{1\dots1}_{a_2} \underbrace{0\dots0}_{a_3}\dots$$
We shall prove that Minkowski's function induces the following correspondence.
\begin{theorem} \label{T:mandel}
Minkowski's question mark function $Q(x)$ maps homeomorphically the bifurcation set $\EKU$ onto the
set $\Omega$ of external angles of rays landing on the main cardioid of the Mandelbrot set. In formulas, 
we have 
$$Q(\EKU) = \Omega.$$
\end{theorem}

The connection may seem incidental, but it is an instance 
of a more general correspondence
discovered by the authors in recent years. 
Indeed, the Minkowski map provides an explicit dictionary between sets of numbers defined using continued fractions 
and sets of external angles for certain fractals arising in complex dynamics. 
More precisely, the question mark function:

\begin{enumerate}
 \item maps homeomorphically the \emph{bifurcation set for 
$\alpha$-continued fractions} onto the set of external rays landing on the \emph{real slice} of the boundary of the Mandelbrot set
(see \cite{BCIT} and \cite{thesis}, Theorem 1.1);
\item maps the sets of \emph{numbers of generalized bounded type} defined in \cite{CT2} to the sets of external rays landing 
on the real slice of the boundary of Julia sets for real quadratic polynomials (\cite{thesis}, Theorem 1.4);
\item
conjugates the tuning operators defined by Douady and Hubbard for the real quadratic family to 
tuning operators corresponding to renormalization schemes for the $\alpha$-continued fractions \cite{CT3}.
\end{enumerate}
For an introduction and more details about such correspondence we refer to 
one of the authors' thesis \cite{thesis}. The dictionary proves to be especially useful 
to derive results about families of continued fractions using the large body of
information known about the combinatorics of the quadratic family;
moreover, it can also be used to obtain new results about the quadratic family and the Mandelbrot set 
from the combinatorics of continued fractions (e.g. \cite{thesis}, Theorem 1.6).

Being  intimately connected to the structure of $\QQ$, Farey words
play a distinguished role in several other dynamical, combinatorial or
algebraic problems. To list just a few, we mention: kneading sequences for Lorentz maps \cite{HS, LaMo}, the coding of cutting sequences on the flat torus
\cite{HM} as well as on the hyperbolic one-punctured torus (\cite{KS}, pg. 726-727); the Markov spectrum 
(in particular the Cohn tree, see \cite{Bo}, pg. 201); primitive elements in rank two free groups \cite{GK};
the Burrows-Wheeler transform \cite{MRS}; digital convexity \cite{BLPR}. For more information we also refer 
to the survey \cite{Be} or the books \cite{Fogg} and \cite{BLRS}. 


\subsection{Behaviour of $(a,b)$-continued fractions on the critical line.}
We conclude the introduction by explaining in more detail the results 
of \cite{KU1, KU2} and how they relate to the present paper. 
For further details, see also section \ref{S:NatExt}.


In  \cite{KU1}, Katok and Ugarcovici consider the two parameter family of continued fraction algorithms induced by the maps
\begin{equation}\label{eq:slow}
f_{a,b}(x):=\left\{
\begin{array}{ll}
x + 1 & \textup{if }x<a\\
-1/x & \textup{if }a\leq x<b\\
x -1 & \textup{if }b\leq x
\end{array}
\right.
\end{equation}
where 
the parameters $(a,b)$ range in a closed subset $\mathcal{P}$ of the plane.
The segment 
$$C:=\{(a,b) \ : \ b-a=1, \ b\in [0,1]\}$$
is a piece of the boundary of $\mathcal{P}$, and the first return map of  $f_{b-1,b}$ on the interval
$[b-1,b)$ coincides with the map $K_b$ we are going to study (these maps are also 
mentioned in \cite{KU1} under the name ``Gauss-like maps'' and denoted $\hat{f}_{b-1, b}$).

Katok and Ugarcovici also consider a closely related family
$(F_{a,b})_{(a,b) \in \mathcal{P}}$ of maps of the plane: each
$F_{a,b}$ has an attractor $D_{a,b}\subset \RR^2$ such that $F_{a,b}$
restricted to $D_{a,b}$ is invertible and it is a geometric
realization of the natural extension of $f_{a,b}$.  They also show
that for most parameters in $\mathcal{P}$ the attractor $D_{a,b}$ has
{\it finite rectangular structure}, meaning that it is a finite union
of rectangles. 
Moreover, all exceptions to this property belong to a Cantor
set $\tilde{\mathcal{E}}$ which is contained in the critical line $C$ and whose 1-dimensional Lebesgue measure is zero.

It turns out that the set $\EKU$ we are considering is just the projection of the set $\tilde{\mathcal{E}}$
onto the second coordinate, up to a countable set.
Making explicit the structure of $\EKU$ allows us to prove that it is not just a zero measure set, but it also has zero Hausdorff dimension.

\subsection*{Structure of the paper}
In section \ref{S:farey} we shall start defining Farey words and 
establishing the properties which are needed to describe the combinatorial dynamics of the $(a,b)$-continued fractions.
Moreover, we define the \emph{binary bifurcation set} $\EB$ and show how it is related to the main cardioid 
of the Mandelbrot set. 

In section \ref{S:rcf} we then recall basic facts about continued fractions and then define 
the \emph{runlength map} $RL$ which passes from binary expansions to continued fraction expansions; then we shall see
how the properties of Farey words translate into the properties of the \emph{qumtervals} which we shall define.
Then, we shall apply all these properties to the case of $(a, b)$-continued fractions; in section \ref{S:matching} 
we determine the combinatorial dynamics of the orbits of $\alpha$ and $\alpha-1$, thus proving that the 
matching condition holds on each qumterval. 

Then, in section \ref{S:entropy} we shall draw consequences for the entropy
function, proving Theorem \ref{T:main}: indeed, we prove that the Cantor set $\EKU$ 
has Hausdorff dimension zero (Proposition \ref{Hdim}) and $h$ is H\"older continuous, so it can be extended to a monotone function 
across the Cantor set. 
Finally, in section \ref{S:NatExt} we shall combine the 
previous properties with the construction of the attractors given in \cite{KU1} to prove Theorem \ref{T:more}.
For the sake of readability, the proofs of some technical lemmas will be postponed to the appendix.









\subsection*{Acknowledgements}
We wish to thank H. Bruin, P. Majer, C.G. Moreira and I. Ugarcovici for useful conversations.
The pictures of the Mandelbrot and Julia sets are drawn with the software {\tt mandel} of W. Jung.

\section{Farey words and dynamics} \label{S:farey}




We shall start by constructing the set of Farey words and establishing the properties which are needed
in the rest of the paper. Many of these results appear in various sources, for instance in the books
\cite{Be, BLRS, Fogg, Lot}. For the convenience of the reader, and in order to set up the notation for the rest of the paper, 
we shall give a fairly self-contained treatment.

\subsection{Alphabets and orderings}
An \emph{alphabet} $\mathcal{A}$ will be a finite set of symbols, which we shall call 
\emph{digits}. 
Given an alphabet $\mathcal{A}$, we shall denote as $\mathcal{A}^n$ the set of words of length $n$, 
as $\mathcal{A}^\mathbb{N}$ the set of infinite words, and as $\mathcal{A}^\star := \bigcup_{n \geq 0} \mathcal{A}^n$ the set of 
\emph{finite words} of arbitrary length. 
If $w$ is a finite word, then the symbol $\overline{w}$ will denote the infinite word given by infinite repetition 
of the word $w$.

If $w  = (\epsilon_1,...,\epsilon_\ell)\in \mathcal{A}^*$, we shall denote 
as $|w|$ the \emph{length} of the word $w$, i.e. the number of digits; 
moreover, if we fix a digit $a \in \mathcal{A}$, the symbol $|w|_a$
will denote the number of digits in the word which are equal to $a$.
Moreover, given a word $w=(\epsilon_1,...,\epsilon_\ell)\in \mathcal{A}^*$, 
we define its \emph{transpose} to be the word ${}^t w$ with  
$${}^t w:=(\epsilon_\ell,...,\epsilon_1); $$
a word which is equal to its transpose is called {\em palindrome}. 
Moreover, we define the \emph{cyclic permutation} operator $\tau$ to act on the word 
$w =  (\epsilon_1,...,\epsilon_\ell)$ as 
$$\tau w:=(\epsilon_2,...,\epsilon_\ell,\epsilon_1).$$

A total order $<$ on the alphabet $\mathcal{A}$ induces, for each $n$, 
a total order on the set $\mathcal{A}^n$ of words of length $n$ by 
using the lexicographic order, and similarly it induces a total order 
on the set $\mathcal{A}^\mathbb{N}$ of infinite words.
We shall extend this order to a (partial) order on the set $\mathcal{A}^*$
of finite words by defining that 
$$u < v \quad \textup{ if } \quad uv < vu.$$
Note that it is not difficult to check that, if $u, v \in \mathcal{A}^*$, 
then the inequality $u<v$ is equivalent to 
$\bar{u}<\bar{v}$ (this fact also proves that $<$ is an order relation).

Finally, we shall also define the stronger partial order relation $<<$ on the set $\mathcal{A}^*$
of finite strings by saying that 
$$u<<v$$ 
if there exist a prefix $u_1$ of $u$ and a prefix $v_1$ of $v$ with $|u_1] = |v_1|$ and such that 
$u_1 < v_1$. 
Note that $u << v$ implies $u < v$, and moreover that any infinite word beginning with $u$ is smaller 
than any infinite word beginning with $v$.

In the following we will mainly be interested in the binary alphabet $\mathcal{A} := \{0, 1\}$, 
with the natural order $0 < 1$.
For $\epsilon\in \{0,1\}$, we also define the \emph{negation} 
operator  $\check{\epsilon}:=1-\epsilon$, which can be extended digit-wise to binary words: 
if $w=(\epsilon_1, ..., \epsilon_\ell) \in \{0, 1\}^*$, we define $\check{w}:=(\check{\epsilon_1},...,\check{\epsilon_\ell})$.


Every infinite word $w=(\epsilon_1, \epsilon_2, ...) \in \{0, 1\}^\mathbb{N}$ also corresponds to the unique real value in $[0,1]$
which has $w$ as its binary expansion, which will be denoted by $.w := \sum_{k =1}^\infty \epsilon_k 2^{-k};$
the same is true for finite binary words in  $\{0,1\}^*$, which correspond to dyadic rationals.



\subsection{Farey words}
We are now ready to define one of the main ingredients of the paper, namely the set of \emph{Farey words}.
As we shall see, several equivalent definitions can be given; we shall start with a recursive definition.

For each integer $n \geq 0$, we shall construct a list $F_n$ of finite words in the alphabet $\{0, 1\}$, 
called \emph{Farey list} of level $n$.
Let us start with $F_0 := (0, 1)$ the list consisting of the two one-digit words.
For each $n$, the next list $F_{n+1}$ is obtained by inserting between two consecutive words $v$, $w$ in the list 
$F_n$ the concatenation $vw$. In formulas, if 
$F_n = (w_1, \dots, w_k)$ with each $w_i$ a finite word, then the next list is 
$F_{n+1} = (v_1, \dots, v_{2k-1})$ with 
$$\begin{array}{ll}
v_{2i-1} :=  w_i & \textup{ for } 1 \leq i \leq k \\
v_{2i} := w_i w_{i+1} & \textup{ for } 1 \leq i \leq k-1.
\end{array}
$$

\begin{definition}
The set of Farey words $FW$ is the union of all Farey lists:
$$FW:=\bigcup_{n\geq 0}F_n.$$
\end{definition}
As an example, the first few Farey lists are
$$
\begin{array}{l}
F_0 = (0,1) \\
F_1= (0,01,1) \\
F_2= (0,001,01,011,1) \\ 
F_3= (0,0001,001,00101,01,01011,011,0111,1)
\end{array}
$$
and all their elements are Farey words. Note that each $F_n$ contains $2^n + 1$ elements.
A Farey word will be called \emph{non-degenerate} if it has more than one digit: we shall denote the set of non-degenerate 
Farey words as $FW^\star = FW \setminus \{ 0, 1 \}$. 

These words are also sometimes called {\it Christoffel words}, as in the book \cite{BLRS}, or 
{\it standard words} as in \cite{MRS}.


\begin{lemma} \label{incr}
Each Farey list $F_n$ is strictly increasing.
\end{lemma}

\begin{proof}
Since $0 < 1$, then the list $F_0$ is ordered increasingly. By induction, we just need to prove 
that given two words $u < v$, then we have 
$$u < uv < v.$$
Indeed, by definition $u < v$ means that 
\begin{equation} 
uv < vu.
\label{uvvu}
\end{equation}
By prefixing $u$ on both sides of the inequality, we get $uuv < uvu$, that we can
interpret as $u(uv) < (uv)u$, hence by definition of the order we have $u < uv$. Similarly, by adding $v$ 
on both sides at the end of \eqref{uvvu} we get 
$uvv < vuv$, which by definition implies $uv < v$.

\end{proof}

Note moreover that each Farey word is naturally equipped with a \emph{standard factorization}; 
indeed, if $w$ is a Farey word, let $n$ be the smallest integer for which $w$ belongs to $F_n$; 
by definition, the word $w$ is generated in the iterative construction as a concatenation $w = w_1 w_2$
where $w_1$ and $w_2$ belong to the level $F_{n-1}$. Thus, the decomposition 
$w = w_1 w_2$ will be called the standard factorization of $w$. In the appendix we shall prove the following characterization:
\begin{proposition}\label{P:standard.factorization}
Given $w\in FW$, let us consider a decomposition $w=w'w''$ where $w', w''$ are non-empty words; then the following conditions are equivalent:
\begin{enumerate}
\item $w'$ and $w''$ are Farey words;
\item $w=w'w''$ is the standard factorization of $w$.
\end{enumerate} 
\end{proposition}

We shall now construct a natural correspondence between the set of Farey words and the set 
of rational numbers between $0$ and $1$. 
Given a word $w \in \{0, 1\}^*$, let us define the rational number $\rho(w)$ to be the 
ratio between the number of occurrences of the digit $1$ and the total length of the word:
$$\rho(w) := \frac{|w|_1}{|w|}.$$
Clearly, $0 \leq \rho(w) \leq 1$. Moreover, we have the following correspondence:

\begin{proposition} \label{fareybij}
The map $\rho : FW \to \mathbb{Q} \cap [0, 1]$ is a bijection between the set of Farey words and 
the set of rational numbers between $0$ and $1$.
\end{proposition}

In the rest of this section we shall prove the Proposition \ref{fareybij}, and meanwhile establish more properties of Farey words. 
In particular, we shall see how to construct an inverse of $\rho$, i.e. to produce a Farey word given a rational number.




Let $r := \frac{p}{q} \in [0, 1]$ be a rational number, 
with $(p, q) = 1$, and consider the one-dimensional torus $\mathbb{R}/\mathbb{Z}$, with the marked point $x_0 = 0$. 
Let $C_q := \{ x \in \mathbb{R}/\mathbb{Z} \ : \ qx \not\equiv 0  \mod 1  \}$.
For each $x \in C_q$, we shall define the binary word $\Phi_r(x) \in \{ 0, 1 \}^q$ by using the dynamics 
of the circle rotation
$$R_r(x) := x + r \mod 1.$$
The word $\Phi_r(x)$ will be constructed as follows: starting at $x$, we successively apply the rotation $R = R_r$ 
and each time we write down $1$ if we cross the $x_0$ mark, and $0$ otherwise. 
More precisely, we define $\Phi_r(x) := (\epsilon_1, \dots, \epsilon_q)$, 
where, for each $k$ between $1$ and $q$, the $k^{th}$ digit $\epsilon_k$ is given by 
$$\epsilon_k := \left\{ \begin{array}{ll} 
                         0 & \textup{if }x_0 \notin (R^{k-1}(x), R^k(x)] \\
			 1 & \textup{if }x_0 \in (R^{k-1}(x), R^k(x)].                       
\end{array}\right.$$
It is immediate to check that one can also write the formula
\begin{equation} \label{altdef}
\epsilon_k = \lfloor x+ k r \rfloor - \lfloor x+ (k-1)r \rfloor \qquad \textup{ for }1 \leq k \leq q. 
\end{equation}
Note that the map $\Phi_r$ intertwines the rotation with the cyclic permutation $\tau$, i.e. 
$$\Phi_r \circ R_r = \tau \circ \Phi_r.$$


\begin{lemma} \label{increase}
The map $\Phi_r$ is (weakly) increasing, in the sense that if $0 < x < y < 1$, then $\Phi_r(x) \leq \Phi_r(y) $.  
\end{lemma}

\begin{proof}
If $x < y$, then for each $k$ we have $\lfloor x + kr \rfloor \leq \lfloor y + kr \rfloor$. Thus, either for each $0 \leq k \leq q$ 
we have $\lfloor x + kr \rfloor = \lfloor y + kr \rfloor$ (in which case $\Phi_r(x) = \Phi_r(y)$) or there exists a minimum $k \leq q$ 
such that $\lfloor x + kr \rfloor < \lfloor y + kr \rfloor$. In the latter case,  $\Phi_r(x) < \Phi_r(y)$ in lexicographical order.
\end{proof}

Moreover, the map $\Phi_r$ is constant on connected components of $C_q$; we will be particularly interested in the word 
$W_r$ defined as 
$$W_r := \Phi_r(0^+) = \lim_{x \to 0^+} \Phi_r(x).$$

\begin{lemma} \label{rightinv}
The map $W : \mathbb{Q} \cap [0,1] \to \{0, 1\}^\star$ is a right inverse of $\rho$; that is, for each $r \in \mathbb{Q} \cap [0, 1]$ we have 
$$\rho(W_r) = r.$$  
\end{lemma}

\begin{proof}
Since all digits of $W_r$ are either $0$ or $1$, then the number of $1$ digits of $W_r$ is just the sum 
of the digits, so by using equation \eqref{altdef}  we get the telescoping sum: 
$$|W_r|_1 = \epsilon_1 + \dots + \epsilon_q = \sum_{k = 1}^q (\lfloor k r \rfloor - \lfloor (k-1) r \rfloor) = \lfloor q r \rfloor = p$$
so $\rho(W_r) = |W_r|_1/|W_r| = p/q = r$.
\end{proof}

A pair $(r, r')$ of rational numbers $r:= \frac{p}{q}$ and $r' := \frac{p'}{q'}$ with $(p, q) = (p', q') =1$ and $pq'-p'q=1$ 
is called a {\it Farey pair}; the \emph{Farey sum} of a Farey pair
is defined as 
$$ r\oplus r' := \frac{p+ p'}{q+q'}.$$
It is easy to check that $r \oplus r'$ lies in between $r$ and $r'$, that is if 
$r < r'$ we have 
\begin{equation} 
r < r \oplus r' < r'.
   \label{fareysum}
\end{equation}

\begin{lemma} \label{Wconcat}
Let $r, r'$ be a Farey pair with $r < r'$; then we have the identity 
$$W_{r \oplus r'} = W_r W_{r'}$$
where on the right-hand side we mean the concatenation of $W_r$ and $W_{r'}$.
\end{lemma}

\begin{proof}[Proof of Proposition \ref{fareybij}.]
By Lemma \ref{rightinv}, the function $\rho$ is surjective, and moreover its restriction to the set 
$$\textup{Im }W := \{ W_r \ : \ r \in \mathbb{Q} \cap [0, 1] \}$$
is a bijection between $\textup{Im }W$ and $\mathbb{Q} \cap [0,1]$.
Therefore, we just need to show that the set $\textup{Im }W$ coincides with the set $FW$ of all Farey words.
Now, since $W_0 = 0$ and $W_1 = 1$, the elements of the Farey list $F_0 = (0, 1)$ belong to $\textup{Im }W$, 
and note that $(0, 1)$ is a Farey pair. Thus, by induction using Lemma \ref{Wconcat}, for each $n$ the elements 
of the list $F_n$ belong to $\textup{Im }W$, so all Farey words belong to $\textup{Im }W$.
Since it is well-known that every rational number can be obtained from $0$ and $1$ by taking successive Farey sums 
of Farey pairs, then $W_r$ is a Farey word for any rational numbers $r \in [0, 1]$, and the claim is proven.
\end{proof}

Note moreover that the above $\rho$ is a bijection between the tree of Farey words and the tree of Farey fractions; 
on one side, the operation is the concatenation of strings, while on the other side it is the Farey sum.







For $w=(\epsilon_1, ..., \epsilon_\ell)\in \{0,1\}^*$ we set
$${}^\vee w:=(\check{\epsilon}_1,\epsilon_2,...,\epsilon_\ell), \ \ w^\vee:=(\epsilon_1,...,\epsilon_{\ell-1},\check{\epsilon}_\ell).$$
We shall now see Farey words have many symmetries, arising from the symmetries of the dynamical system $R_r$.

\begin{proposition}\label{P:simmetrie}
If $w = W_r$ is a Farey word, then:
\begin{enumerate}
\item[(a)] the word ${}^t\check{w}$ is still a Farey word: in particular, 
$$W_{1-r} = {}^t \check{w};$$ 
\item[(b)] moreover, we have the identity 
$$\Phi_{1-r}(0^-) = \check{w}$$
\item[(c)] and
$$\Phi_r(0^-) = {}^\vee w^\vee={}^tw;$$
\item[(d)] 
both ${}^\vee w$ and $w^\vee$ are palindromes;
\item[(e)]
finally, we have 
$$ {}^t w < {}^\vee w.$$


\end{enumerate}
\end{proposition}

As an example, let us pick $w = W_{2/5} = 00101$. One can check that ${}^\vee w = 10101$ and $w^\vee = 00100$ are both 
palindromes, and ${}^\vee w^\vee = 10100$ equals the transpose of $w$. Finally, the word ${}^t\check{w} = 01011$ is also a Farey word
($ = W_{3/5}$).

\begin{proof}

(a) Let us note that considering the rotation $R_{1-r}$ instead of $R_r$ is equivalent to inverting the 
direction (clockwise or counterclockwise) of the rotation. Thus, for each $x \in C_q$, the first $q+1$ elements of the orbit 
of $x$ under $R_r$ are the same as the first $q+1$ elements of the orbit of $x$ under $R_{1-r}$, but the order of visit is reversed
(in symbols, $R^k_r(x) \equiv R^{q-k}_{1-r}(x) \mod 1$ for $0 \leq k \leq q$), which proves the claim.

(b) This identity relies on the fact that the circle is symmetric under reflection $\sigma(x) := -x \mod 1$; 
indeed, for each $x$ the orbit of $x$ under $R_r$ is the reflection of the orbit of $1-x$ under $R_{1-r}$ 
(in symbols, $R^k_r(x) \equiv - R^k_{1-r}(-x) \mod 1$), while the marked point $x_0 = 0$ is fixed by $\sigma$.

(c) The first equality follows by noting that the iterates $R_r^k(0)$ encounter a discontinuity 
of the function $\lfloor \cdot \rfloor$ if and only if $k \equiv 0 \mod q$; thus, 
changing the starting point $x$ from $0^+$ to $0^-$ only affects the first and last digits of $\Phi_r(x)$. 
For the second equality, denote $v := \Phi_{1-r}(0^+) = {}^t\check{w}$; we have 
by (b) and then (a) 
$$\Phi_r(0^-) = \check{v} = {}^t w.$$

(d) follows immediately from (c): indeed we have 
$${}^t({}^\vee w) = ({}^t w)^\vee = ( {}^\vee w^\vee)^\vee = {}^\vee w$$
where the first and third equalities are elementary, and the second one uses (c); 
a completely analogous proof works for $w^\vee$.

(e) Applying (c), and using the fact that the last digit of each (non-zero) Farey word is $1$, we get
$${}^tw = {}^\vee w^\vee < {}^\vee w.$$
\end{proof}

It will be crucial in the following to study the ordering of the set of cyclic 
permutations of a given Farey word. The essential properties are contained in the following lemma.

\begin{lemma} \label{cyclic}
Let $w = W_r$ be a Farey word, and consider the set 
$$\Sigma(w) := \{ \tau^k w \ : \ k \in \NN\}$$ 
of its cyclic permutations. Moreover, let $w = w_1 w_2$ be the standard factorization 
of $w$, and let $q_1 := |w_1|$, $q_2 := |w_2|$. Then the following are true:
\begin{enumerate}
\item the smallest cyclic permutation of $w$ is $w$ itself (i.e., $\min \Sigma(w) = w$);
\item
the second smallest cyclic permutation of $w$ is 
$$\tau^{q_1} w = w_2 w_1;$$
\item
the largest cyclic permutation of $w$ is 
$$\tau^{q_2} w = {}^tw.$$
\end{enumerate}
\end{lemma}

\begin{proof}
Let us start by noting that, if $w = \Phi_r(x)$, then the set of cyclic permutations of $w$ is given by 
$$\{ \tau^k w \ : \  0 \leq k < q \} = \{ \Phi_r(R_r^k(x)) \ : \ 0 \leq k < q \};$$
moreover, by Lemma \ref{increase}, the order in the above set is the same as the order in the set 
$$S_r(x) := \{ \{x + kr\} \ : \ 0 \leq k < q \}.$$
Thus, the smallest cyclic permutation of $w = W_r$ corresponds to the smallest possible value of $\{ kr \}$, which 
is attained for $k = 0$, hence by $w = \Phi_r(0^+)$ itself, proving (1).

Moreover, let $w = w_1 w_2$ be the standard factorization of $w$. Then by definition we have 
$w = W_r$, while $w_1 = W_{r_1}$ and $w_2 = W_{r_2}$, in such a way that $(r_1, r_2)$ is a Farey pair, 
with $r_1 < r_2$ and $r := r_1 \oplus r_2$.
Note now that, writing $r_1 = \frac{p_1}{q_1}$ and $r_2 = \frac{p_2}{q_2}$, we have $p_2 q_1 - p_1 q_2 = 1$
by the definition of Farey pair, hence we can write
\begin{equation} \label{mod1}  
q_1 (p_1+p_2) \equiv 1 \mod (q_1 + q_2).
\end{equation}
Thus, the second smallest element of $S_r(0)$ is attained for $k = q_1$, hence 
the second smallest element of $\Sigma(w)$ is $\tau^{q_1} w = w_2 w_1$, proving (2).

Finally, the largest element of the set $\Sigma(w)$ is $\tau^k w$, where $k$ is such that  
$\{ k r\} = 1-\frac{1}{q}$; thus, the corresponding word is $\Phi_r(0^-)$, which equals ${}^t w$ by 
Proposition \ref{P:simmetrie} (c). Moreover, from equation \ref{mod1} one also gets
$$q_2 (p_1+p_2) \equiv -1 \mod (q_1 + q_2)$$
hence $\{ q_2 r\} = 1-\frac{1}{q}$, so the largest element of the set $\Sigma(w)$ is $\tau^{q_2} w$.
\end{proof}

Let us now state one more consequence of the previous lemma, in terms of ordering of subsets of the circle.
Recall the \emph{doubling map} $D: \RR/\ZZ \to \RR/\ZZ$ is defined as $D(x) := 2x \mod 1$.
We say that a finite set $X \subseteq S^1$ has \emph{rotation number} $r = \frac{p}{q} \in \mathbb{Q}$ if 
it is invariant for the doubling map, and the restriction of $D$ to $X$ is conjugate to the circle rotation $R_r$ via an orientation-preserving homeomorphism of $S^1$. 
More concretely, this means that if we write the elements of $X$ in cyclic order as $X = (\theta_0, \theta_1, \dots, \theta_{q-1})$ 
with $0 \leq \theta_0 < \theta_1 < \dots < \theta_{q-1} < 1$, then we have for each index $i$
$$D(\theta_i) = \theta_{i+p} \ \mod q.$$
The proof of the previous lemma also yields the following (uniqueness follows from \cite{Gol}, Corollary 8):

\begin{lemma} \label{rotnumber}
Let $w = W_r$ a Farey word. Then the set 
$$C(w) = \{ 0.\overline{\tau^k w} \ : \ 0 \leq k \leq q-1 \} \subseteq S^1$$
is the unique subset of $S^1$ which has rotation number $r$ for the doubling map. 
\end{lemma}
\noindent For an example, if $w = 00101$, then $C(w) = (\frac{5}{31}, \frac{9}{31}, \frac{10}{31}, \frac{18}{31}, \frac{20}{31})$ (see also 
Figure \ref{F:cardioid}).

Recall that a word $w\in \{0,1\}^*$ which is minimal (with
respect to lexicographic order) among all its cyclic permutations is also called a \emph{Lyndon word}, 
hence property (e) of Lemma \ref{cyclic} can be paraphrased as saying that every Farey word is a Lyndon word (but not viceversa: e.g., 
$0011$ is a Lyndon word but not a Farey word).
Let us recall that all Lyndon words of length greater than 1 begin with the digit $0$ and end with the digit 
$1$; moreover one has the following (see \cite{Lot})
\begin{proposition}\label{P:lyndon}
If $w=ps$ is a Lyndon word (in particular, if $w$ is a Farey word), then $w<<s$.
\end{proposition}


\subsection{Substitutions}

Another way to generate Farey words is by substitutions. 
Given a pair of words $U=\fect{u_0}{u_1} \in \{0,1\}^* \times \{0,1\}^*$ we
can define the \emph{substitution operator} associated to $U$ to be the 
operator acting on $\{0,1\}^*$ (or on $\{0,1\}^{\NN}$) as 
$$w= (\epsilon_1,\epsilon_2,...) \mapsto
(u_{\epsilon_1}, u_{\epsilon_2},...);$$
the action of $U$ on $w$ will be denoted by $w \star U$. Let us note that if $u_0<u_1$ then the
operator is order preserving while if $u_0>u_1$ it is order reversing;
moreover, the negation operator can be obtained as the 
substitution associated to 
$V:=\fect{(1)}{(0)}$.
We can also extend the substitution operator to pairs of words: 
if $U=\fect{u_0}{u_1}$  and $W=\fect{w_0}{w_1}$ let us define 
$U \star W:=\fect{u_0 \star W}{u_1 \star W}$; in this way we get 
the following associativity property, that for each word $w$ we have
\begin{equation}\label{eq:associative}
(w \star U)\star W = w \star (U \star W).
\end{equation}
Finally, the substitution and transposition operators are compatible, in the sense that 
\begin{equation}\label{eq:transpose}
{}^t(w\star U)={}^t w \star {}^tU  \ \ \ \mbox{where } \ {}^tU:=\fect{{}^t u_0}{{}^tu_1}.\end{equation}

It turns out that one can produce all (non-degenerate) Farey words by successive iteration of two substitution 
operators, starting with the word $w_0 = (01)$. Namely, let us define 
the two substitution operators
$$ U_0: 
\left\{
\begin{array}{l}
0\mapsto 0\\
1\mapsto 01
\end{array}
\right.
\ \ 
U_1:
\left\{
\begin{array}{l}
0\mapsto 01\\
1\mapsto 1.
\end{array}
\right.
$$ 
It is not difficult to realize that the action of $U_0$ and $U_1$ preserves the set of Farey words.
More precisely, let us set $F_n^*:=F_n \setminus\{(0),(1)\}$,
$$FW_0:=\{w\in FW^* : |w|_0>|w|_1\} \ \ FW_1:=\{w\in FW^* : |w|_0<|w|_1\},$$
 and $F_n^\epsilon:=F_n \cap FW_\epsilon$ (note that, by Proposition \ref{fareybij}, for each $n \geq 2$ one has 
$F_n^\star = F_n^0 \cup \{(01)\} \cup F_n^1$).
We can now formulate the following lemma. 
\begin{lemma}\label{L:u.tuning}
For each $\epsilon = 0, 1$, the operator $U_\epsilon:F^*_n \to F_{n+1}^\epsilon$ is a bijection.
\end{lemma}

\begin{proof}
By induction, using the fact that substitution operators $U_0,U_1$ preserve the lexicographical order and respect concatenation, in the sense that
$$(vw)\star U_\epsilon = (v\star U_\epsilon)(w \star U_\epsilon).$$
\end{proof}


\begin{proposition}\label{P:t.char}
For all $n\geq 1$, the following characterization holds:
$$F_n^*=\left\{ (01)\star  U_{\epsilon_1}\star ... \star U_{\epsilon_\ell}  \ : \ \epsilon_k \in \{0,1\}, \ \  0 \leq \ell < n \right\}.$$
\end{proposition}
 \begin{proof}
Again by induction: the base of the induction ($n = 1$) is true, and if the claim holds at level $n$ then
we get, by Lemma \ref{L:u.tuning}, for each $\epsilon = 0, 1$
$$
\begin{array}{l}
F_{n+1}^\epsilon=U_\epsilon(F_{n}^*)= \left\{ (01)\star 
U_{\epsilon_1}\star ... \star U_{\epsilon_\ell} \star U_{\epsilon} \ : \ \epsilon_k \in \{0, 1\}, \ 0 \leq \ell < n  \right\}.
\end{array}
$$
Thus, using the fact 
$$F_{n+1}^*=F_{n+1}^0\cup F_{n+1}^1 \cup \{(01)\} $$
the claim follows.
\end{proof}

\subsection{Farey words, kneading theory and external angles}


Denote $[a, b]$ 
the closed interval of the circle from $a$ to $b$, with positive orientation.
Let us define the \emph{binary bifurcation set} $\EB$ as 
$$\EB:=\{x \in [0,1/2]: D^k(x) \in [x, x+1/2] \ \ \forall k \in \NN \}.$$
The set $\EB$ is a closed subset of the interval $[0, 1/2]$ and it has no interior as we will see. Let us point out that the only dyadic rationals which belong to $\EB$ are $0$ and $1/2$.
Moreover, we shall see that the connected components of the complement of $\EB$ are 
canonically labelled by Farey words. Indeed, if $w\in\{0,1\}^*$
is a Farey word we set  $I_w:=(a^-, a^+)$ with
$$\begin{array}{l}
a^+:=0.\overline{w}  \\ 
a^-:= 0.\overline{{}^t w} -1/2. \\ 
\end{array}
$$
For instance, if $w = 00101$ then $a^+ = 0.\overline{00101} = \frac{5}{31}$ and $a^- = \frac{9}{62}$. 
We have the following properties.

\begin{proposition}\label{L:book}
With the notation above we have
\begin{enumerate}
\item $a^\pm \in \EB$;
\item if $x\in I_w$ then $x \notin \EB $;
\item for each Farey word $w$, the length of $I_w$ is 
$$|I_w|=\frac{1}{2(2^n-1)}$$ with $n=|w|$;
\item each $I_w$ is a connected component of $[0,1/2]\setminus \EB$; moreover, we have 
$$[0,1/2]\setminus \EB = \bigcup_{w \in FW} I_w;$$
\item the Hausdorff dimension of $\EB$ is zero. 
\end{enumerate}
\end{proposition}
Various similar constructions of this set appear in the literature; in particular, the above properties
are proven by Bullett and Sentenac \cite{BS}; thus, for the convenience of the reader we still give a complete 
proof of the Proposition, but we postpone it to the appendix. 
The set $\EB$ also appears in the kneading theory for Lorentz maps: indeed, it is the one-dimensional 
projection of the two-dimensional set of all kneading invariants for Lorentz maps (see \cite{HS, LaMo}). 

Finally, the following lemma will be needed in the last section.

\begin{lemma} \label{L:difference}
If $x\in \EB \cap [0,1/6)$ then the following limit is infinite:
$$\lim_{\delta \to 0}\inf\{|w|_0-|w|_1 \ : \ I_w\subset[x-\delta,
    x+\delta] \}=+\infty.$$
\end{lemma}

\begin{proof}
Indeed, note that setting $\rho := \rho(w)$ we can rewrite 
$|w|_0 - |w|_1 = |w|(1 - 2 \rho)$; thus, if $w_n$ is a sequence of Farey words 
such that $I_{w_n}$ converge to $x \in \EB$, then $\rho(w_n)$ tends to a finite number, 
which is $< \frac{1}{2}$ if $x \in [0, 1/6)$, while $|w|$ tends to infinity, hence 
the liminf of the product is infinite.
\end{proof}


\medskip

Let us now highlight a connection between the combinatorics of Farey words and the symbolic coding
of rays landing on the main cardioid of the Mandelbrot set. We shall start by recalling few standard 
facts in complex dynamics: for an account, we refer to \cite{Mi} and references therein.

Let us consider the family of quadratic polynomials $f_c(z) := z^2 + c$ with $c\in \mathbb{C}$.
Recall the \emph{filled Julia set} $K(f)$ of a polynomial $f(z)$ is the set of points with bounded orbits:
$$K(f) := \{ z \in \mathbb{C} \ : \ \sup_n |f^n(z)| < \infty \}.$$
If $K(f)$ is connected, then its exterior is conformally isomorphic to a disk, hence it can be uniformized by a unique map
$\Phi : \hat{\mathbb{C}}\setminus \overline{\mathbb{D}} \to \hat{\mathbb{C}}\setminus K(f)$
with $\lim_{z\to \infty}|\Phi(z)| = \infty$ and $\lim_{z \to \infty}\Phi(z)/z >  0.$
For each $\theta \in \mathbb{R}/\mathbb{Z}$, the \emph{external ray at angle $\theta$} is the set 
$$R(\theta) := \{ \Phi(\rho e^{2 \pi i \theta}) \ | \ \rho > 1\}.$$ 
The ray $R(\theta)$ is said to \emph{land} if $\lim_{\rho \to 1^+} \Phi(\rho e^{2 \pi i \theta})$ exists 
(and then it is a point on the boundary of $K(f)$). The \emph{Julia set} $J(f)$ is the topological 
boundary of $K(f)$. By Carath\'eodory's theorem, if the Julia set is locally connected, then all rays land.
The map $f_c$ has two fixed points (which coalesce if $c = \frac{1}{4}$); we shall call $\beta$\emph{-fixed point}
the fixed point where the external ray of angle $\theta = 0$ lands, and $\alpha$\emph{-fixed point} the other fixed point.
A fixed point $z_0$ is called \emph{indifferent} when the derivative $f_c'(z_0)$ has modulus $1$ (the derivative $f_c'(z_0)$
is usually called the \emph{multiplier} of the fixed point).

In parameter space, let us recall the \emph{Mandelbrot set} $\mathcal{M}$ is the set of parameters $c$ 
for which the orbit under $f_c$ of the critical point $z = 0$ is bounded: 
$$\mathcal{M} := \{ c \in \mathbb{C} \ : \ \sup_{n}|f_c^n(0)| < \infty \}.$$
The set $\mathcal{M}$ also equals the set of parameters $c \in \mathbb{C}$ for which the Julia set of $f_c$ 
is connected.
Just as the Julia sets, the Mandelbrot set admits a unique uniformizing map $\Phi_M : \hat{\mathbb{C}} \setminus \overline{\mathbb{D}} \to 
\hat{\mathbb{C}} \setminus \mathcal{M}$ such that $\lim_{z\to \infty}|\Phi_M(z)| = \infty$ and $\lim_{z \to \infty}\Phi_M(z)/z >  0$.

Let us define the \emph{main cardioid} \Heart \ of the Mandelbrot set as 
$$\textup{\Heart} := \{ c \in \mathbb{C} \ : \ f_c(z) \textup{ has an indifferent fixed point} \}.$$
A simple computation shows \Heart \ can be parametrized as $c = \frac{1}{2} e^{2 \pi i \theta} - \frac{1}{4} e^{4 \pi i \theta}$
where $\theta \in [0, 1]$; in this parametrization, for each $c \in $\Heart \ the map $f_c$ has multiplier $e^{2 \pi i \theta}$ 
at the $\alpha$-fixed point.
Let $\Omega$ denote the set of angles of external rays landing on the main cardioid of the Mandelbrot set:
$$\Omega := \{ \theta \in \mathbb{R}/\mathbb{Z} \ : \ R_M(\theta) \textup{ lands on \Heart }\}.$$
The following proposition makes precise the connection between Farey words and the set of rays landing on the cardioid.

\begin{proposition} \label{P:cardioid}
Let $r = \frac{p}{q} \in \mathbb{Q}\cap (0, 1)$, and $w = W_r$ the corresponding Farey word. 
Let us now define the pair of angles $(\theta^-,\theta^+)$ as
$$
\begin{array}{lll}
\theta^- & = & 0.\overline{\tau ({}^t w)} \\
\theta^+ & = & 0.\overline{\tau w}
\end{array}
$$
and let $c \in $\ \Heart \ denote the parameter on the main cardiod for which the $\alpha$-fixed point of 
$f_c$ has multiplier $e^{2 \pi i r}$. Then we have the following properties:
\begin{enumerate}
\item in the Julia set of $f_c$, the set of external rays landing at the $\alpha$-fixed point is 
the set 
$$C(w) = \{ 0.\overline{\tau^k w} \ : \ 0 \leq k \leq q-1 \}$$
whose binary expansions are all cyclic permutations of $w$;
\item in parameter space, the pair of angles of external rays $(\theta^-, \theta^+)$ lands on the main 
cardioid at the parameter $c$;
\item we have the identity
$$\Omega = 2 \EB.$$
\end{enumerate}

\end{proposition}

As an example, if $r = \frac{2}{5}$ then $w = 00101$, and $\theta^- = 0.\overline{01001} = \frac{9}{31}$, 
while $\theta^+ = 0.\overline{01010} = \frac{10}{31}$. In the dynamical plane, the set of rays landing at the 
$\alpha$-fixed point is $C(w) = (\frac{5}{31}, \frac{9}{31}, \frac{10}{31}, \frac{18}{31}, \frac{20}{31})$.

\begin{figure}
\fbox{\includegraphics[scale=0.13]{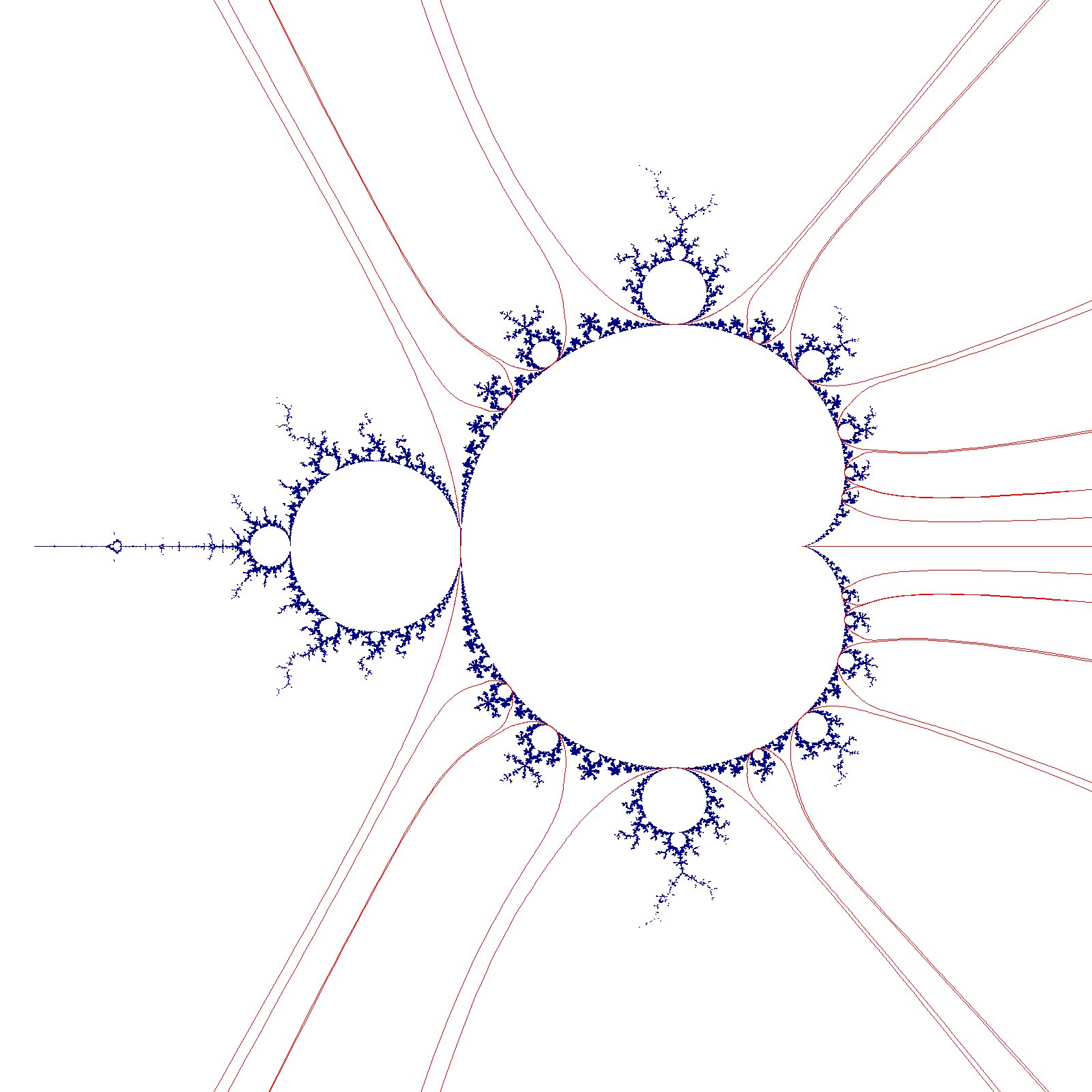}} 
\fbox{\includegraphics[scale=0.13]{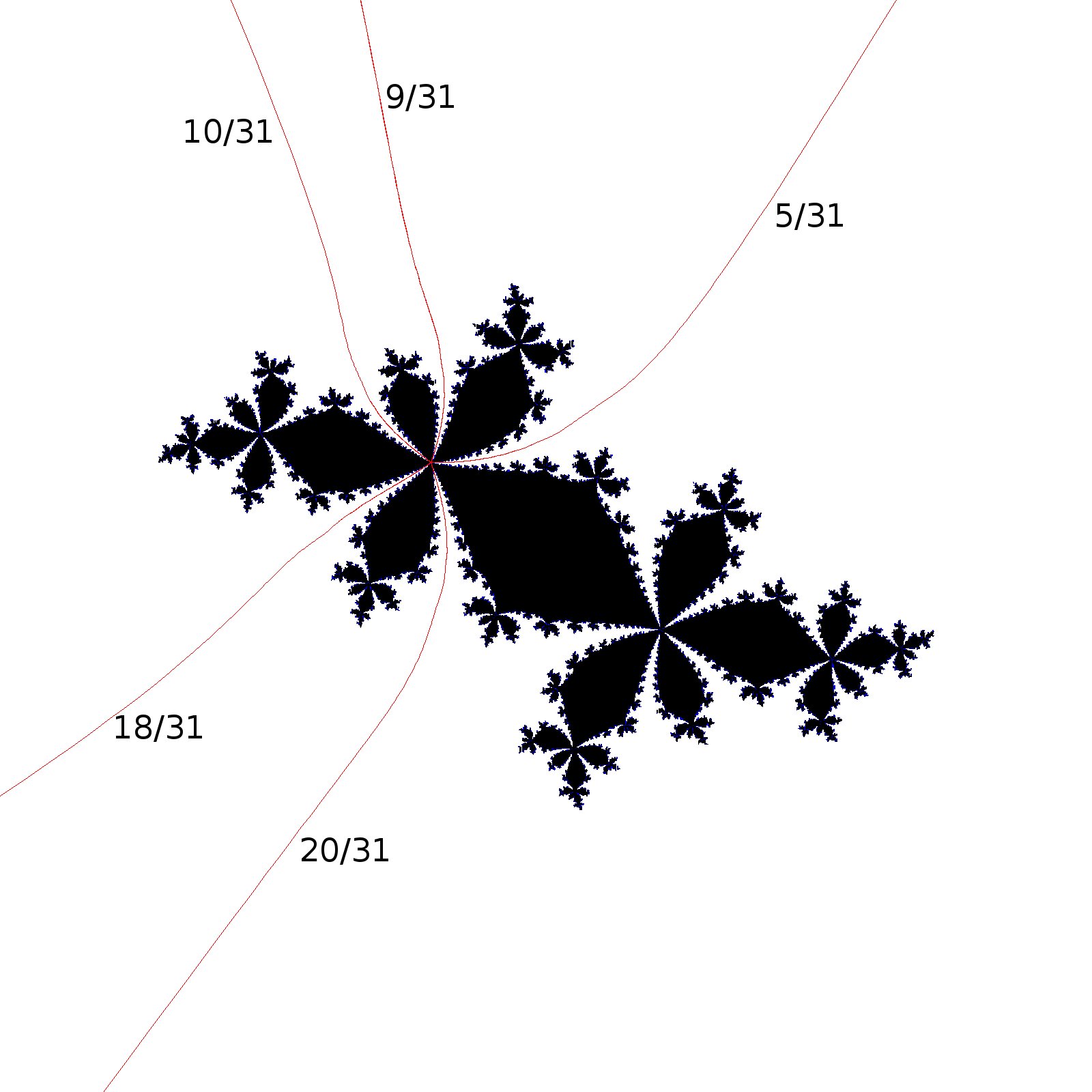}}

\caption{Left: the set $\Omega$ of external rays landing on the main cardioid of the Mandelbrot set.
Right: the set $C(w)$ of external rays landing on the $\alpha$-fixed point of a Julia set for the center 
of a hyperbolic component tangent to the main cardioid (here, the rotation number is $r=2/5$ and the corresponding Farey word is $w = 00101$.)}
\label{F:cardioid}
\end{figure}

\begin{proof}
(1) Let $c \in $\Heart \ be the parameter for which the map $f_c$ has an indifferent fixed point of multiplier $e^{2 \pi i r}$. 
It is known that its Julia set $J(f_c)$ is locally connected, hence all external rays land in the dynamical plane of $f_c$, 
and the \emph{landing map} 
$L(\theta) : \mathbb{R}/\mathbb{Z} \to J(f_c)$ defined as $L(\theta):= \lim_{\rho \to 1^+} \Phi(\rho e^{2\pi i \theta})$
is a continuous semiconjugacy between the doubling map and $f_c$, that is we have the commutative diagram (see also \cite{Mi}):
$$\xymatrix{\mathbb{R}/\mathbb{Z} \ar^D@(ul,ur) \ar^L[r] & J(f_c) \ar^{f_c}@(ul, ur) }$$
Let $S := L^{-1}(\alpha)$ the set of angles of external rays landing at the $\alpha$-fixed point. 
The map $f_c$ permutes the rays landing at $\alpha$, and preserves their cyclic orientation in the plane. Moreover, 
since the multiplier of $f_c$ at $\alpha$ is $e^{2\pi i r}$, then the set $S$ has rotation number $r$ under the doubling map, 
hence by Lemma \ref{rotnumber} the set $S$ equals $C(w)$.

To pass to the statement in parameter space, note that it is known that the pair of rays landing on 
$c$ in parameter space corresponds to the elements of $S$ which delimit a sector which (in the dynamical plane) 
contains the critical value.
Now, the set $S^1 \setminus S$ is the union of $q$ (connected) arcs, and the doubling map permutes their endpoints: 
hence, $D$ maps each arc of length smaller than $1/2$ homeomorphically to its image, and there is a unique arc $\ell_0$ 
of length at least $1/2$. Since the map $f_c$ is a local homeomorphism away from its critical point, then the 
component in the dynamical plane corresponding to $\ell_0$ must contain the critical point.
Now note that by Lemma \ref{cyclic} (1) and (3), the arc  $\ell = (0.\overline{{}^tw}, 0.\overline{w})$ is a connected component of 
$S^1 \setminus S$, and by Proposition \ref{L:book} and the definition of $\EB$, the length of $\ell$ is more than $1/2$, 
so it must be $\ell = \ell_0$ the one which contains the critical point. 
As a consequence, the arc which contains the critical value is delimited by taking the forward image of the endpoints
of $\ell_0$; thus, it is the arc $\ell_1 = (0.\overline{\tau ({}^t w)}, 0.\overline{\tau w}) = 
(\theta^-, \theta^+)$ and claim (2) is proven.
 
As for the last statement, the previous construction implies the correspondence
$\Omega \cap \mathbb{Q} = 2 (\EB \cap \mathbb{Q})$; 
claim (3) follows by taking closures, as it is known that the set of angles of rays landing on the main cardioid is 
the closure of the set of rational angles of rays landing on the main cardioid (see \cite{Hub}, Corollary 4.4).
\end{proof}





\section{Regular continued fraction expansions} \label{S:rcf}
Let us first fix some notation regarding the classical continued
fractions expansions. 
 Any irrational number admits a unique infinite
continued fraction expansion, which will be denoted as 
$$ x=a_0+\cfrac{1}{a_1+\cfrac{1}{a_2+ ...}} = [a_0; a_1, a_2,...]$$
with $a_k \in \ZZ$ $\forall k$, and $a_k \geq 1 \ \forall k \geq 1$. 
Moreover, any rational value $r$ admits exactly two finite expansions; 
indeed, we can write 
$$r = [a_0; a_1, \dots, a_n] = [a_0; a_1, \dots, a_n-1, 1]$$
with $a_n \geq 2$. 
Any nonempty string of positive integers 
$S = (a_1, \dots, a_n)$ defines a rational value $r = [0; a_1, \dots, a_n] \in (0,1]$, 
which we will sometimes denote as $r = [0; S]$.

We then define the {\em right conjugate} of $S$ to be the only string
$S'$ which defines the same rational value as $S$, i.e. such that $[0;S']=[0;S]$. For
instance $(3,1,3)'=(3,1,2,1)$ and viceversa (conjugation is
involutive, and affects only the last one or two digits). We also define the
{\em left conjugate} $'\!S$ of a (finite or infinite) string $S$ in a similar
way, just acting on the leftmost digits: that is, if $S = (a_1, a_2, \dots)$
we define 
$$'\!S:= 
\left\{ \begin{array}{ll}
(1, a_1-1, a_2, \dots) & \textup{if } a_1 \geq 2 \\
(1 + a_2, a_3, \dots) & \textup{if } a_1 = 1. 
\end{array}
\right.
$$
Thus, the left conjugate of $(3,1,3)$ will be $'(3,1,3)=(1,2,1,3)$. It is not difficult to check
that this manipulation on strings translates into the map
$\sigma:[0,1]\to [0,1]$ defined as $\sigma(x):=1-x$ on the side of
continued fraction expansions, namely for any string of positive integers we have 
\begin{equation} \label{eq:sigmax}
\sigma([0;S])=[0; {'\!S}].
\end{equation}
 Another operation on
strings we shall often use in the following is the operator $\partial$ defined on (finite or infinite) strings as
$$\partial (a_1,a_2,...) :=  \left\{ \begin{array}{ll} 
                                      (a_1-1,a_2,...) & \textup{if }a_1 > 1 \\
                                      (a_2,...) & \textup{if }a_1 = 1. \\
	                                       \end{array}
\right.$$
We shall sometimes also use the transposition: the transpose string of $S=(a_1,...,a_\ell)$ is the string ${}^t S=
(a_\ell,...,a_1)$. 
Finally, if $S$ is a finite string of positive integers we will denote by $q(S)$ the denominator 
of the rational number  whose c.f. expansion is $S$, i.e. such that $\frac{p(S)}{q(S)} = [0; S]$ with $(p(S), q(S)) = 1$, $q(S) > 0$. 

Let us also recall the well-known estimate 
\begin{equation}\label{eq:supermult}
q(S)q(T) \leq q(ST)\leq 2q(S)q(T).
\end{equation}
Moreover, we define the map $f_S : x \mapsto S \cdot x$, which
corresponds to appending the string $S$ at the beginning of the continued fraction
 expansion of $x$. 
That is, if $S=(a_1,...,a_n)$ we can write, by identifying matrices with M\"obius transformations, 
\begin{equation}\label{eq:bblocks} S \cdot x := \matr{0}{ 1}{ 1}{ a_1} \matr{0}{ 1}{ 1}{ a_2}...  \matr{0}{
  1}{ 1}{ a_n} \cdot x.
\end{equation}
It is easy to realize that concatenation of
strings corresponds to composition, namely  
$ (ST)\cdot x = S
\cdot(T\cdot x)$; moreover the map $f_S$ is increasing
if $|S|$ is even, decreasing if $|S|$ is odd.
The image of $f_S$ is a \emph{cylinder set} 
$$I(S) := \{ x = S \cdot y,\ y \in [0, 1] \}$$ which is a closed interval with endpoints
$[0; a_1, \dots, a_n]$ and $[0; a_1, \dots, a_n +1]$. 
The map $f_S$ is a contraction of the unit interval, and it is easy to see that 
\begin{equation} \label{eq:contraction}
\frac{1}{4 q(S)^2} \leq |f'_S(x)| \leq \frac{1}{q(S)^2} \qquad \qquad \forall x \in [0,1]
\end{equation}
and the length of $I(S)$ is bounded by
\begin{equation} \label{eq:cylsize}
\frac{1}{2 q(S)^2} \leq |I(S)| \leq \frac{1}{q(S)^2}.
\end{equation}

Given two strings of positive integers $S = (a_1, \dots, a_n)$ and $T = (b_1, \dots, b_n)$ of equal length, let us define the 
\emph{alternate lexicographic order} as
$$S < T \ \textup{if } \exists k\leq n \textup{ s.t. } a_i = b_i \ \forall 1 \leq i \leq {k-1} \textup{ and } \left\{ \begin{array}{ll} a_n < b_n & \textup{if }n\textup{ even} \\
                                                                                                                       a_n > b_n & \textup{if }n\textup{ odd}. 
                                                                                                                      \end{array} \right.$$
The importance of such order lies in the fact that given two strings of equal length $S < T \textup{ iff } [0; S] < [0; T]$. 
In order to compare quadratic irrationals with periodic expansion, the following \emph{string lemma} (\cite{CT}, Lemma 2.12) is useful:
 for any pair of strings $S$, $T$ of positive integers, we have the equivalence
\begin{equation} \label{eq:stringlemma}
ST < TS \Leftrightarrow [0; \overline{S}] < [0; \overline{T}].
\end{equation}
The order $<$ is a total order on the strings of positive integers of fixed length; to be able to compare strings of different lengths 
we define the partial order 
$$S << T \quad \textup{if }\exists i \leq \min\{|S|, |T|\} \textup{ s.t. }S_1^i < T_1^i$$
where $S_1^i = (a_1, \dots, a_i)$ denotes the truncation of $S$ to the first $i$ characters. 
Let us note the following basic properties:

\begin{enumerate}
\item if $|S| = |T|$, then $S < T$ iff $S << T$;
\item if $S, T, U$ are any strings, $S << T \Rightarrow SU << T, S << TU$; 
\item if $S << T$, then $S \cdot z < T \cdot w$ for any $z, w \in (0, 1)$.
\end{enumerate}

\subsection{Farey legacy}

We shall now see how to construct, using continued fractions, an irrational number
given a binary word; this way, starting from the set of Farey words we shall define the fractal subset $\EKU$ 
of the interval, and establish its properties from the properties of Farey words we obtained in the previous sections.

Indeed, let us define the \emph{runlength map} $RL$ 
to be the map which associates to a (finite or infinite) binary word $w$ the string of positive integers which records the size of 
blocks of consecutive equal digits: namely, if 
$$w = \underbrace{0\dots0}_{a_1} \underbrace{1\dots1}_{a_2} \dots$$
we set
$$RL(w) := (a_1, a_2, \dots).$$  
For instance, $RL(0001001001)=RL(1110110110)=(3,1,2,1,2,1)$; 
note that $RL$ is a two-to-one map ($RL(w)=RL(\check{w})$), 
 but it is strictly increasing when restricted to words  beginning with the digit $0$.
If $S=RL(w)$ for some $|w|>1$ then  
\begin{equation}\label{eq:vee}
RL({}^\vee w)= {'\!S}, \ \ \ \ \ \ RL(w^\vee)=S'.
\end{equation}
Note also that, if $w=w_0w_1$ and the last digit of $w_0$ is different from the first of $w_1$, then 
one has 
$$RL(w)=RL(w_0)RL(w_1)$$ 
(note this is always the case when $w_i$ are non-degenerate Farey words).  
For the runlength string of Farey
words some more nice properties hold:

\begin{lemma}\label{L:fareylegacy}
Let $w\in FW^*$ be a Farey word and $S:=RL(w)$. Then the length $|S|$ is even and
\begin{enumerate}
\item[(i)] there exists an integer $a\geq 1$ and a Farey word $f = (\epsilon_1,...,\epsilon_n)$ such that
one can write
$$S=B_{\epsilon_1}\dots B_{\epsilon_n}$$
 with  
$$\begin{array}{lll}
B_0 =  (a+1, 1),  & B_1 = (a, 1) & \textup{ if }w \in FW_0 \\ 
\end{array}$$
or 
$$\begin{array}{lll}
                  B_0 = (1, a), & B_1 = (1, a+1) & \textup{ if }w \in FW_1.
                 \end{array}$$
The Farey word $f$ is unique as long as $w \neq (01)$; since it plays a central role in the following, it will be referred 
to as the \emph{Farey structure} of the string $S$. 

\item[(ii)] The runlength   
 of the Farey word ${}^t\check{w}$ is 
$$RL( {}^t\check{w})= {'\!S'}={}^t S;$$
\item[(iii)] if $S=(a_1,\dots,a_\ell)$ and $1\leq k < \ell/2$, we set $P_k:=(a_1,\dots,a_{2k})$, $S_k:=(a_{2k+1},\dots,a_\ell)$,
(so that $S=P_kS_k$), then
\begin{equation}\label{eq:elyndon}
S<<S_k; 
\end{equation} 
\item[(iv)] using the same notation as above, if $w\in FW_0$  then 
\begin{equation}\label{eq:partial}
S_k P_k<<\partial S.
\end{equation}
 \end{enumerate}
\end{lemma}

\begin{proof}

{\bf (i).} 
Recall that $FW^* = FW_0 \cup (01) \cup FW_1$. Clearly, for $w = (01)$ we have $S = RL(w) = (1, 1)$, so 
$a = 1$ and we can choose $f = (0)$ or $f = (1)$. 
Let us now assume $w\in FW_0$. Then by Proposition \ref{P:t.char} and Lemma \ref{L:u.tuning} we can 
write 
\begin{equation} \label{fact0a}
w=f \star U_1 \star U_0^a 
\end{equation}
 for some
$f=(\epsilon_1, \dots, \epsilon_\ell) \in FW$ and $a\geq 1$; on the other
hand $U_1 \star U_0^a=\fect{(0^{a+1}1)}{(0^a1)}$ and $RL(0^{a}1)=(a,1)$ so,
calling $B_0:=(a+1,1)$ and $B_1:=(a,1)$ we get that $S$ is the
concatenation $B_{\epsilon_1}\dots B_{\epsilon_\ell}$. 
Note that, since the image of $U_0$ is contained in $FW_0$ and the image of $U_1$
is contained in $FW_1$ (Lemma \ref{L:u.tuning}), then the factorization of equation \eqref{fact0a} is unique, hence also the Farey structure 
of $S$ is unique.
The case $w\in
FW_1$ is analogous.

{\bf (ii).} The second claim is an immediate consequence of equation \eqref{eq:vee} and the fact that ${}^\vee w^\vee= {}^t w$, 
together with the fact that $RL(\check{w}) = RL(w)$.

{\bf (iii).} It follows from Proposition \ref{P:lyndon}, and the fact that the runlength map preserves the strong order $<<$ 
when restricted to words which begin with $0$. 

{\bf (iv).}
By unwinding the definitions it is not hard to see that, if $w \in FW_0$, we can write the identities
$$
\begin{array}{l}
S= RL(f \star U_1 \star U_0^a)\\
\partial S= RL(({}^\vee f) \star U_1 \star U_0^a)\\
S_kP_k=RL((\tau^k f) \star U_1 \star U_0^a).
\end{array}$$
Moreover, by Proposition \ref{P:simmetrie} (e) and Lemma \ref{cyclic} (3) we have 
$$
 \ \ \  {}^\vee f >> {}^t f \geq \tau^k f  \ \ \ \forall k
$$
hence the claim follows from the fact that both the substitution operator $f \mapsto f \star U_1 \star U_0^a$ 
and the runlength map (when restricted to words beginning with $0$) are order-preserving. 

\end{proof}

\subsection{Qumtervals} \label{S:qumtervals}
For $x\in [0,1/2]$ we shall consider the map $\phi : [0, 1/2] \to [0, 1]$ induced by runlength as follows:
if $x=\sum_{j\geq 1} \epsilon_j 2^{-j}$ is the binary expansion of $x$, with $\epsilon_j \in \{0, 1\}$, then we define 
$\phi(x)$ to be the number with continued fraction
$$\phi(x):=[0;RL(\epsilon)]$$
where $RL(\epsilon)$ is the runlength of the sequence $(\epsilon_j)_{j \geq 1}$. 
This map is certainly well-defined  for those values of $x$ which admit a unique (and infinite) binary
expansion; in fact, it also extends continuously to dyadic rationals, since the two binary expansions 
of a dyadic rational are mapped onto two continued fraction expansions of the same rational. 
For instance, if $x = \frac{3}{8} = 0.011$ then $\phi(x) = [0; 1, 2] = \frac{2}{3}$; on the other hand, 
we can write $\frac{3}{8} = 0.010\overline{1}$, which maps to $[0; 1, 1, 1, \infty] = \frac{2}{3}$.
 It is not difficult to check that this map is a homeomorphism between $[0,1/2]$ and $[0,1]$.
The inverse of $\phi$ is essentially \emph{Minkowski's question mark} function $Q : [0, 1] \to [0,1]$
sending $x = [0; a_1, a_2, \dots]$ to 
$$Q(x) := 0.\underbrace{0\dots0}_{a_1-1}\underbrace{1\dots1}_{a_2}\dots.$$
In fact, one has for each $x \in [0, 1/2]$
\begin{equation}\label{MinkInv}
Q(\phi(x)) = 2 x.
\end{equation}
For properties of Minkowski's question mark function, we refer to \cite{Sa}.

\begin{definition}\label{D:qum}
Given $w\in FW^*$ we will call \emph{ qumterval of label $w$} the interval $J_w$ which is the image under $\phi$ 
of the interval $I_w$ appearing in Lemma \ref{L:book}, that is 
$$J_w:= \phi(I_w).$$
\end{definition}
Note that, if $S:=RL(w)$, then by Lemma \ref{L:book} the endpoints of $J_w$ are given by  
\begin{equation}\label{eq:qumterval}
J_w=(\alpha^-, \alpha^+)  \ \ \ \mbox{ with } 
\begin{array}{l}\alpha^+ := \phi(a^+)=[0;\overline{S}]\\ 
\alpha^- := \phi(a^-)=[0; S' \overline{{}^tS}].
\end{array}
\end{equation}
As an example, the Farey word $w = 001$ yields $S = (2, 1)$, hence $\alpha^+ = [0; \overline{2, 1}] = \frac{\sqrt{3}-1}{2}$ and 
$\alpha^- = [0; 3, \overline{1, 2}] = 2 - \sqrt{3}$.
The rational value $r:=\phi(.w)=[0;S]$ is the (unique!) rational value in $J_w$ with least denominator, and 
will be called the {\it pseudocenter} of $J_w$ (see \cite{CT} for more general properties of the pseudocenter
of an interval).
Note that, by using equation \eqref{eq:sigmax} and Lemma \ref{L:fareylegacy} (ii), the left endpoint $\alpha^-$ can also be described 
by the property 
\begin{equation} \label{eq:leftendpoint}
1 - \alpha^- = [0; \overline{{}^t S}].
\end{equation}
We also define $\EKU:=\phi (\EB)$. In this way we get
$$[0,1]\setminus \EKU=\bigcup_{w\in FW^*} J_w.$$
Let us point out that, since $\phi$ is a bijection between dyadic rationals  and rationals in $[0,1]$, we have $\EKU\cap \QQ=\{0,1\}$.
Theorem \ref{T:mandel} now follows immediately from our setup. 

We now have the tools to prove one of the results stated in the introduction.

\medskip

\noindent \textbf{Proof of Theorem \ref{T:mandel}.}
Since by definition $\EKU = \phi(\EB)$, then by using equation \eqref{MinkInv} and Proposition \ref{P:cardioid} (3)
we get 
$$ Q(\EKU) = Q(\phi(\EB)) = 2\EB = \Omega.$$ 
\qed

\subsection{Thickening $\mathbb{Q}$}

We shall now perform an alternative construction of $\EKU$ which
is not essential for the main results of this paper, but it is useful for a comparison with the results in \cite{CT}.
Given any rational value $r\in (0,1)$, let us consider its
\cfe of even length $r=[0;S]$; then set $\beta(r):= [0;\overline{S}]$
and
$$\tilde{J}_r:=(\sigma \beta(\sigma r), \beta(r)).$$ Since
$\beta(r)>r$ and $\sigma$ is order reversing, we can easily see that
the $\tilde{J}_r$ is an open interval containing $r$, and in fact $r$ is
the pseudocenter of $\tilde{J}_r$. 
For all $w\in FW^*$ we have that $J_w=\tilde{J_r}$ for
$r=\phi(.w)$. Indeed qumtervals have the following maximality
property (which will be proven in the appendix):
\begin{proposition}\label{P:thick}
For any $r'\in \QQ\cap (0,1)$ there is a Farey word $w\in FW^*$ such that $\tilde{J}_{r'}\subset J_w$. 
\end{proposition}
As a consequence of the above proposition one gets the identity 
\begin{equation} \label{E:EKU}
\EKU=[0,1]\setminus \bigcup_{r\in \QQ \cap (0,1)} \tilde{J}_r.  
\end{equation}

\begin{figure}
\centering
\fbox{\includegraphics[scale=0.45]{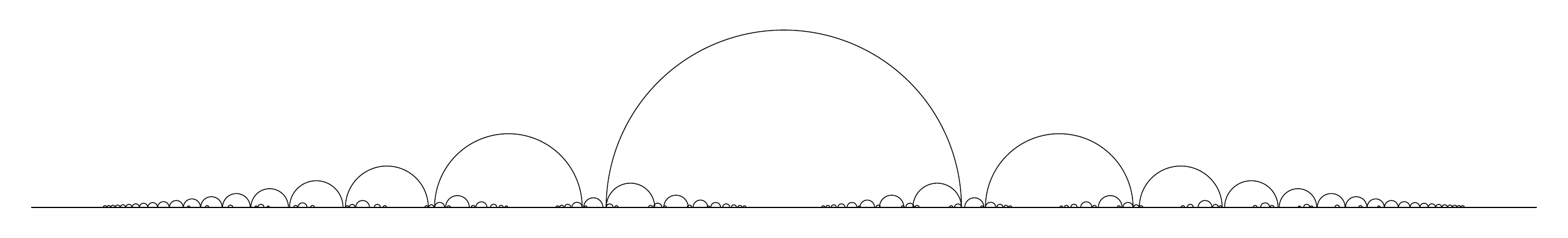}}
\caption{The quadratic intervals $\tilde{J}_r$. Each interval is represented by a half-circle 
with the same endpoints. The intervals which are maximal with respect to inclusion are precisely 
the connected components of the complement of the bifurcation set $\EKU$.}
\label{F:EKU} 
\end{figure}

Let us now compute the dimension of $\EKU$.
\begin{proposition} \label{Hdim}
The Hausdorff dimension of $\EKU$ is zero:
$${\rm H.dim}\ \EKU=0.$$ 

\end{proposition}

\begin{proof}
We shall actually prove the stronger statement that for each $N \geq 2$ 
the set $\EKU \cap [\frac{1}{N+1}, \frac{1}{N}]$ has zero box-counting dimension: 
the claim then follows since the box-counting dimension is an upper bound for the Hausdorff 
dimension.
Fix $N \geq 2$, set
$$C_N:=\left\{w\in FW_0
  \ : J_w\cap \left[\frac{1}{N+1}, \frac{1}{N}\right]\neq \emptyset\right\}$$ 
and consider  the geometric $\zeta$-function defined by
$$\zeta_N (s) := \sum_{w\in C_N} |J_w|^s. 
$$
Since the abscissa of convergence of the series $\zeta_N$ coincides with the upper box dimension
of $\EKU \cap [\frac{1}{N+1}, \frac{1}{N}]$ (see \cite{F}, pg. 54), it is enough to prove that 
the above series converges for any $s > 0$.
Now, it is not hard to prove that, for all $N\geq 2$, one has
\begin{equation} \label{eq:direct}
|J_w|<2 \ b^{|w|}  \ \ \ \ \forall w\in C_N, \ \ \ \mbox{ where } b:= N^{-\frac{2}{N+1}}.
\end{equation}
Indeed, it is easy to check that $$|r-\beta(r)|=|S \cdot 0 - S \cdot \beta(r)|\leq  \sup |f'_S| \beta(r) \leq \frac{1}{q(S)^2}  $$
where the last inequality is a consequence of equation \eqref{eq:contraction}; since an analogous estimate holds for the distance between $r$ and the left endpoint, one gets
\begin{equation}\label{eq:qvsJ_w}
 |J_w|<\frac{2}{q(S)^2}.
\end{equation}
On the other hand, if $w\in C_N$ then
$S=RL(w)$ is a concatenation of $n$ blocks of the type $B_0:=(N,1)$ or
$B_1:=(N-1,1)$, where $n (N+1)<|w|$. Thus we get that
$$q(S)=q(B_{\epsilon_1}...B_{\epsilon_n})\geq q(B_1^n)\geq q(B_1)^n$$
and since $q(B_1)\geq N$ we get $q(S)\geq N^\frac{|w|}{N+1}$; thus \eqref{eq:direct} follows from this last estimate and equation \eqref{eq:qvsJ_w}. 
Since $\#\{w \in FW \ : \  |w|=k\}\leq k$, the estimate \eqref{eq:direct} implies that $\zeta_N$ is dominated by the sum $2^s\sum_1^{\infty}k b^{sk}$, therefore it
converges for all $s>0$, proving the claim.

\end{proof}


\medskip

Let us conclude this section by comparing the bifurcation set $\EKU$ with the 
bifurcation set (or \emph{exceptional set}) $\EE_N$ for Nakada's $\alpha$-continued 
fraction transformations (see \cite{CT}).
By comparing equation \eqref{E:EKU} with the definition\footnote{In \cite{CT}, the bifurcation 
set is simply denoted by $\mathcal{E}$, and its complement is denoted by $\mathcal{M}$. See also \cite{CT3}, section 3.}  of $\EE_N$ from \cite{CT}, one can easily check the inclusion
\begin{equation} \label{E:inclusion}
\EKU\cap [0,1/2] \subset \EE_N.
\end{equation}
Note that the inclusion is strict (and actually, the Hausdorff dimension of $\EE_N$ is $1$, 
while the dimension of $\EKU$ is $0$).

Since both sets in \eqref{E:inclusion} are related to sets of rays landing in the Mandelbrot set, 
it is intersting to see what our dictionary tells us when we transport the previous inclusion
to the world of complex dynamics. 
First, using Theorem \ref{T:mandel}, the Minkowski question mark $Q$ maps $\EKU$ homeomorphically to the 
set $\Omega$ of external angles of rays landing on the main cardioid. 
Meanwhile, by the main theorem of \cite{BCIT}, the set $\EE_N$ is related to the set of rays landing on the 
real slice of the Mandelbrot set. Indeed, if we let $\mathcal{R}$ be the set of external angles of rays whose 
impression intersects the real slice of the Mandelbrot set, then we have 
the homeomorphism (\cite{BCIT}, Theorem 1.1)
$$\psi(\EE_N) = \mathcal{R} \cap [1/2, 1)$$ 
where $\psi(x) := \frac{1}{2} + \frac{Q(x)}{4}$.
Thus we have the following commutative diagram, where $i$ is the inclusion map:
$$\xymatrix{\EKU \cap [0, 1/2] \ar^{Q}[d]  \ar[r]^{\quad i} &
  \EE_N \ar^{\psi}[d]\\ \Omega\cap[0, 1/2] \ar@{.>}[r]^{\quad T} & \mathcal{R} }$$
As a consequence, the map $T(\theta) := \psi(Q^{-1}(\theta))$, which can be just expressed as 
$$T(\theta) = \frac{1}{2} + \frac{\theta}{4}$$
maps the set of rays landing on the upper half of the main cardioid into the set of real rays, i.e.
$$T(\theta) \subseteq \mathcal{R}$$ 
for each $\theta \in \Omega \cap [0, 1/2]$.
This fact is known in the folklore as ``Douady's magic formula'' (see \cite{Bl}, Theorem 1.1).




\section{Matching intervals for continued fractions with $SL(2,\ZZ)$-branches.} \label{S:matching}
Let us return to the maps $K_\alpha :[\alpha -1, \alpha] \to [\alpha
  -1, \alpha]$ which are defined by $K_\alpha (0)=0$ and
$$
K_\alpha (x)= -\frac{1}{x} - c_\alpha(x),
\ \ \ \ \ \ \ \ c_\alpha(x):=\left \lfloor -\frac{1}{x} +1-\alpha \right \rfloor \in \ZZ.$$ 
The goal of this section is to prove that a matching condition between the orbits of the endpoints
$\alpha$ and $\alpha-1$ is achieved for any parameter which belongs to some qumterval 
(Theorem \ref{T:anatomy} and Corollary \ref{C:weakmatching}). 
In order to formulate the result precisely, we need some notation.

Recall the group 
$PSL(2, \mathbb{Z})$ acts on the real projective line by M\"obius transformations.
Indeed, if $\mathbf{A} = \matr{a}{b}{c}{d}$ and $x \in \mathbb{R} \cup \{ \infty \}$ then we shall write
$\mathbf{A}x := \frac{ax +b}{cx + d}$. 
The group $PSL(2, \mathbb{Z})$ is generated by the two elements $\mathbf{S}$ and $\mathbf{T}$, which are 
represented by the matrices
$${\bf S} := \matr{0}{-1}{1}{0} \qquad {\bf T} := \matr{1}{1}{0}{1}$$
and act respectively as the inversion ${\bf S} x := -1/x$ and the translation ${\bf T}x:=x+1$. 
For any fixed $\alpha \in (0,1)$,  the map $K_\alpha$ is just given by the inversion followed by an integer power 
of the translation  which brings the point back to the interval $[\alpha-1,\alpha]$.
Thus, each branch of $K_\alpha(x)$ is represented by the map
$x \mapsto \mathbf{T}^{-c_\alpha(x)}\mathbf{S}x$. 
Now, in order to keep track of the inverse branches of the powers of $K_\alpha$, 
we shall now use the notation $c_{j,\alpha}(x):=c_\alpha(K_\alpha^{j-1}(x))$ for each positive integer $j$, 
and define the matrices
$${\bf M}_{\alpha, x, \ell}:=\matr{0}{-1}{1}{c_{1,\alpha}(x)} \cdot ... \cdot \matr{0}{-1}{1}{c_{\ell,\alpha}(x)}.$$
Note that these matrices represent the inverses of $K_\alpha$, in the sense that 
$${\bf M}_{\alpha, x, \ell}( K^\ell_\alpha(x)) = x$$
for each $\alpha \in [0, 1], x \in [\alpha-1, \alpha], \ell \in \mathbb{N}$. 
Finally, note that the family $K_\alpha$ possesses the following fundamental symmetry:
the maps $K_\alpha$ and $K_{1-\alpha}$ are measurably conjugate, 
namely one has
 \begin{equation}\label{eq:ssimmetria}
K_\alpha(x)=-K_{1- \alpha}(-x) 
\end{equation}
for all $x \in [\alpha-1,\alpha] \setminus \bigcup_{k \in \mathbb{Z}} \frac{1}{k-\alpha}$ (the countable 
set of exceptions is due to the convention about the floor function).
As a consequence, it is sufficient to study the dynamics for $\alpha \in [0, 1/2]$.

We are now ready to formulate the main result of this section. 



\begin{theorem}\label{T:anatomy}
Let $w\in FW^*$ a Farey word,  $m_0:=|w|_0, m_1:=|w|_1$ and  $J_w$ the corresponding qumterval; 
let moreover $S:=RL(w)$ denote the runlength of $S$ and $r:=[0;S]$ be the pseudocenter of $J_w$.
Then there exist two elements ${\bf M, M'} \in 
PSL(2,\ZZ)$ such that, for all $\alpha \in J_w$  we have the equalities
\begin{equation}\label{eq:cost}
\begin{array}{l}
{\bf M}_{\alpha,\alpha-1,m_0}={\bf M} \\
{\bf M}_{\alpha,\alpha,m_1}={\bf M'}. \\
\end{array}
\end{equation}
Moreover, the following {\em matching condition} holds:
\begin{equation}\label{eq:amatching}
{\bf TM}= {\bf M' ST}^{-1}{\bf S}. 
\end{equation}

\end{theorem}

The matching condition \eqref{eq:amatching} implies the following identification between of the orbits of the 
two endpoints $\alpha$ and $\alpha-1$.

\begin{corollary} \label{C:weakmatching}
For each parameter $\alpha \in J_w$ we have the identity
\begin{equation} \label{eq:weakmatching}
K_\alpha^{m_0+1}(\alpha-1)=K_\alpha^{m_1+1}(\alpha). 
\end{equation}
\end{corollary}

Note that \eqref{eq:cost} implies that the first $m_0$ steps of the (symbolic) itinerary of $\alpha$ is constant
for all $\alpha$ in the same qumterval, and the same is true for the first $m_1$ steps of the orbit of $m_1$.
A condition of this kind is called {\bf strong cycle condition} in \cite{KU1}; 
see Section \ref{S:NatExt} for a more detailed comparison.

As an illustration of Theorem \ref{T:anatomy}, let us consider the case $\alpha \in J_w$ with $w=01$: 
it turns out that for every $\alpha \in J_{01}=(g^2,g)$ the following identity holds:
\begin{equation}\label{eq:cyc01}
K^2_\alpha(\alpha)=K^2_\alpha(\alpha-1) \ \ \ \ \forall \alpha \in (g^2,g).
\end{equation}
Indeed, this is due to the fact that the analytic expression of $K_\alpha$ at the endpoints does not change as $\alpha \in J_{01}$; 
in this simple case in fact we can work out the explicit form of $K_\alpha$, and we get 
$$
K_\alpha(\alpha)= {\bf T}^2{\bf S } \alpha=\frac{2\alpha -1}{\alpha} \ \ \ \ \   K_\alpha(\alpha-1)={\bf T}^{-2}{\bf S }(\alpha-1)=\frac{2\alpha -1}{1-\alpha}.
$$
whence we have $\mathbf{M} = \mathbf{S}\mathbf{T}^2$ and $\mathbf{M}' = \mathbf{S}\mathbf{T}^{-2}$, 
and we can check that 
$$\mathbf{TM} = \matr{1}{1}{1}{2} = \mathbf{M}'\mathbf{S}\mathbf{T}^{-1}\mathbf{S}$$
which is an instance of equation \eqref{eq:amatching}. Note that the essential point is that the matrices 
$\mathbf{M}$
and $\mathbf{M}'$ do not depend on the particular $\alpha$ as long as $\alpha$ belongs to $J_w$; thus, the matching condition
is just an identity between elements of the group $PSL(2, \mathbb{Z})$. However, to different matching intervals $J_w$
there correspond different identities in the group.



The proof of Theorem \ref{T:anatomy} follows from an explicit description of the symbolic orbits of $\alpha$ and $\alpha-1$
in terms of the regular continued fraction expansion of the pseudocenter of $J_w$, as stated in the following proposition.


\begin{proposition}\label{R:mconcatenazione}
Let $w\in FW_0$ (hence $J_w \subset (0,1/2)$), denote $m := |w|_0$ and  $n:=|w|_1$
and let $RL(w) = (a_1,1,...,a_n,1)$ be the associated string of positive integers. 
Then for each $\alpha \in J_w$ we have the identities
\begin{equation}
\begin{array}{c}
{\bf M}_{\alpha, \alpha-1, m} = {\bf M}\\
{\bf M}_{\alpha, \alpha, n} = {\bf M'}
\end{array}
\end{equation}
where the above matrices are constructed as 
\begin{equation}\label{eq:mm'}
\begin{array}{l}
{\bf M}:=({\bf ST}^2)^{a_1} {\bf T} ({\bf ST}^2)^{a_2} {\bf T}...({\bf ST}^2)^{a_{n-1}}{\bf T} ({\bf ST}^2)^{a_n}  \\
{\bf M'}  := {\bf S} {\bf T}^{-a_1-1} {\bf S} {\bf T}^{-a_2-2}... {\bf T}^{-a_{n-1}-2} {\bf S} {\bf T}^{-a_n-2}. 
\end{array}
\end{equation}
\end{proposition}

Let us point out that in the special case $n = 1$ we have $RL(w) = (N, 1)$ for some $N \geq 1$, and 
the above equations must be interpreted as yielding
${\bf M} =({\bf ST}^2)^N$, ${\bf M'} ={\bf ST}^{-N-1}$. 
Moreover, as a consequence of the Proposition, the matrices determining matching conditions 
behave well under concatenation; indeed, if the Farey word $w$ is the concatenation of two Farey words $w'$ and $w''$, then 
the left-hand side of the matching condition on $J_w$ is the concatenation of the left-hand sides 
of the matching conditions of $J_{w'}$ and $J_{w''}$.

Before delving into the core of the proofs of Proposition \ref{R:mconcatenazione} and Theorem \ref{T:anatomy} let us 
make some elementary observations and define some more notation.
The action of both ${\bf S, \ T}$ can be easily expressed in terms of regular 
continued fraction expansion, thus the action of 
$K_\alpha$ on the regular continued fraction expansion of $x$ follows some simple rules. 
Namely, if  $x=[-1;a_1,a_2,a_3,...]\in [\alpha-1,0)$ then one gets the formulas
\begin{equation}\label{eq:rules-}
K_\alpha(x)=\left\{
\begin{array}{ll}
{\bf T}^{-2} {\bf S} x= [-1; a_1-1,a_2,a_3, a_4...] & a_1 >1\\
{\bf T}^{-(a_2+1-\epsilon)} {\bf S} x = [\epsilon; a_3, a_4, ...] & a_1 = 1
\end{array}\right.  \ \ \ \epsilon \in \{-1,0\}
\end{equation}
where to decide whether $\epsilon$ is $-1$ or $0$ one has to check which of these choices returns an element in $[\alpha-1,\alpha]$.
On the other hand, if $x=[0;a_1,a_2,a_3,...]\in (0,\alpha]$ then
\begin{equation}\label{eq:rules+}
K_\alpha(x)= {\bf T}^{a_1+1+\epsilon} {\bf S} x= \left\{
\begin{array}{lll}
  &[\epsilon; 1,a_2-1,a_3,...]& a_2 >1\\
& [\epsilon; 1+a_3, a_4,...] & a_2 = 1
\end{array}\right.  \ \ \ \ \ \epsilon \in\{-1,0\};
\end{equation}
again, the choice between $\epsilon=-1$ and $\epsilon=0$ is forced by the condition that the range of $K_\alpha$ must be $[\alpha-1 ,\alpha]$.
To write some of the above branches of $K_\alpha$ in compact form we shall also use the following 
fractional transformations:
$$\partial^- := {\bf S} {\bf T^{-1}} {\bf S}, \ \ \ \  \partial := {\bf S} {\bf T} {\bf S}.$$
Note that, if $0 < x < 1$, one has $\partial^- x < x$; in fact in terms of regular continued fractions we get 
$\partial^- ([0;a_1, a_2, a_3,...])=[0;1+a_1,  a_2, a_3,...]$, while $\partial$ is the inverse of $\partial^- $
(and is consistent with the previous definition in section \ref{S:rcf}). 
Finally, if $S = (a_1, \dots, a_n)$ and $x \in [0, 1]$, we shall use the \emph{string action} notation
$S \cdot x$ to denote the number whose continued fraction expansion is obtained by appending $S$ at the beginning 
of the continued fraction expansion of $x$; in terms of {\bf S} and {\bf T}, this can be defined as 
\begin{equation}\label{eq:evenl}
S \cdot x := {\bf S}{\bf T}^{-a_1}{\bf S }{\bf T}^{a_2}...{\bf S}{\bf T}^{-a_{2n-1}}{\bf S}{\bf T}^{a_{2n}} x.
\end{equation}




\begin{proof}[Proof of Proposition \ref{R:mconcatenazione}.]
Let $w \in FW_0$, and $\alpha \in J_w$. 
Let us denote $S := RL(w)$ the runlength of $S$ and $r := [0; S]$ the pseudocenter of $J_w$.
Now, by Lemma \ref{L:fareylegacy} we have that $S$ is of the form 
$$S=(a_1,1,a_2,1,...,a_{m},1)$$
 with $a_j\in \{a, a+1\}$, $m_0=\sum_{j = 1}^m a_j$, and $m_1=m$. Moreover for
$1\leq k\leq m$ we define the even prefixes and suffixes of $S$ as
$$P_k:=(a_1,1,a_2,1,...,a_{k},1), \ \ \ \ \ S_k:=(a_{k+1},1,a_{k+2},1,...,a_{m},1).$$
Recall also that by definition the endpoints of $J_w = (\alpha^-, \alpha^+)$ are 
$$\alpha^- = [0; S'\overline{{}^t S}] \qquad \alpha^+ = [0; \overline{S}].$$

{\bf Case A.} Let us first take into account the case $\alpha \in [r, \alpha^+)$.
Then we can write $\alpha:= S\cdot y$ for some $y\in [0,  \alpha^+)$. In this case 
we claim that the orbits of the endpoints $\alpha$ and $\alpha -1$ under $K_\alpha$ 
eventually match, and before getting to the matching point the orbits (and symbolic orbits, 
on the right column) of $\alpha-1$ and $\alpha$ are given by table \ref{eq:trajA}.

\begin{table} 

$$\begin{array}{|rcl|r|}
\hline
 \alpha-1&=& -1+S\cdot y&  \\
 K_\alpha(\alpha-1)&=& -1+\partial S \cdot y&  c_{1,\alpha}=2\\
 K^2_\alpha(\alpha-1)&=& -1+\partial^2 S \cdot y&c_{2,\alpha}=2\\
 ...&&& ...\\
 K^{a_1-1}_\alpha(\alpha-1)&=& -1+\partial^{a_1-1} S \cdot y&c_{a_1-1, \alpha}=2\\
 \maltese \ \ \  K^{a_1}_\alpha(\alpha-1)&=& -1+S_1 \cdot y&     c_{a_1,\alpha}=3 \\ 

K^{a_1+1}_\alpha(\alpha-1)&=& -1+\partial S_1 \cdot y& c_{a_1+1,\alpha}=2\\
...&&&...\\
\maltese \ \ \ K^{a_1+a_2}_\alpha(\alpha-1)&=& -1+ S_2 \cdot y & c_{a_1+a_2,\alpha}=3\\
...&&&...\\
...&&&...\\
\maltese \ \ \ K^{a_1+...+a_{m-1}}_\alpha(\alpha-1)&=& -1+ S_{m-1} \cdot y& c_{a_1+...+a_{m-1},\alpha}=3 \\
...&&&...\\
K^{a_1+...+a_{m}-1}_\alpha(\alpha-1)&=& -1+ \partial^{a_m-1} S_{m-1} \cdot y&c_{a_1+...+a_{m}-1,\alpha}=2 \\
\maltese \ \ \ K^{a_1+...+a_{m}}_\alpha(\alpha-1)&=&  y& c_{a_1+...+a_{m},\alpha}=2\\
\hline 
\alpha&=& S\cdot y& \\
 K_\alpha(\alpha)&=& \partial^- S_1 \cdot y & c_{1,\alpha}=-a_1-1 \\
 K^2_\alpha(\alpha)&=& \partial^{-} S_2 \cdot y& c_{2,\alpha}=-a_2-2\\
 ...&&&...\\
 K^{m-1}_\alpha(\alpha)&=& \partial^{-} S_{m-1} \cdot y&  c_{m-1,\alpha}=-a_{m-1}-2 \\
 K^{m}_\alpha(\alpha)&=&\partial^{-} y=\frac{y}{y+1} & c_{m,\alpha}=-a_{m}-2\\
\hline
\end{array}$$
\caption{Orbits of $\alpha$ and $\alpha-1$, for $\alpha \in [r, \alpha^+)$.}
\label{eq:trajA}
\end{table}

One can go from one line to the following 
just using the rules  \eqref{eq:rules-} (for the upper part) or \eqref{eq:rules+} (for the lower part). 
So we only have to check  that
at each stage we actually get a value which lies in the interval $[\alpha-1, \alpha]$.

As far as the orbit of $\alpha-1$ is concerned, all items in the list except for the last one are negative, 
so we just have to check
that we never drop below $\alpha-1$; since the operator $\partial$ increases the value of its 
argument (i.e. $\partial x \geq x$), it is sufficient to check 
the iterates of $K_\alpha$ of order $a_1+...+a_k$ (with $k\in\{1,...m\}$): the corresponding lines are 
marked by the symbol $\maltese $. 
That is, we need to check the following: 

\begin{enumerate}
 \item 
$-1 + S_k \cdot y \geq \alpha -1 \quad \textup{for all } k \in \{1, \dots, m-1 \} $
\item 
$y \leq \alpha$.
\end{enumerate}
(1) is true by Lemma \ref{L:fareylegacy}-(iii): indeed, we have the inequality $S_k  >> S$, from which follows that 
$$-1 + S_k \cdot y \geq -1 + S \cdot y = -1 + \alpha$$
as needed.
Now, since by construction 
$S\cdot y=\alpha < \alpha^+$ and the map $x \mapsto S\cdot x$ is increasing with a fixed point at $\alpha^+$, we have that 
$y\leq S\cdot y =\alpha$, which proves (2).

Checking that the values in the lower part of table \ref{eq:trajA} are actually in $[\alpha-1, \alpha]$
is slightly more tricky. We need to prove that $\partial^-S_k \cdot y \leq S\cdot y = \alpha$ for  $k \in \{1,...,m-1\}$; 
as a matter of fact by Lemma \ref{L:fareylegacy}-(iv) one has
$$ S_k P_k << \partial S,$$
which implies, since $P_k$ is a prefix of the continued fraction expansion of $\alpha^+$,
$$S_k \cdot \alpha^+ \leq \partial S\cdot y.$$
Since the map $x \mapsto S_k \cdot x$ is increasing we then get 
$S_k \cdot y \leq S_k \cdot \alpha^+ < \partial S\cdot y$, 
which implies, by applying $\partial^-$ to both sides of the equation, 
$\partial^-S_k \cdot y  < S\cdot y = \alpha$.
Since $0<\partial^-y<y<S\cdot y=\alpha$,  we get the last step for free.

\smallskip

{\bf Case B.} We must now settle the case $\alpha \in (\alpha^-,r]$. Let us recall that $'\!S'={}^t S$, so  that we must have
 $\alpha= S'\cdot y$ for $0\leq y \leq [0; \overline{{}^tS}]$ or, which is equivalent,  $\sigma \alpha = {}^tS \cdot y$,  $0\leq y \leq [0; \overline{{}^tS}]$.  
In this case we claim that the orbits  of the endpoints, before reaching the matching point, are given by table \ref{eq:trajB}


\begin{table}
$$\begin{array}{|rcl|r|}
\hline
 \alpha-1 &=& -1+S'\cdot y & \\
 K_\alpha(\alpha-1) &=& -1+\partial S' \cdot y & c_{1,\alpha}=2\\
 K^2_\alpha(\alpha-1) &=& -1+\partial^{2} S' \cdot y & c_{2,\alpha}=2 \\
 ...&&& ...\\
 K^{a_1-1}_\alpha(\alpha-1) &=& -1+\partial^{a_1-1} S' \cdot y & c_{a_1-1,\alpha}=2 \\
\maltese \ \ \  K^{a_1}_\alpha(\alpha-1)&=& -1+S'_1 \cdot y & c_{a_1,\alpha}=3\\
K^{a_1+1}_\alpha(\alpha-1)&=& -1+\partial S'_1 \cdot y& c_{a_1+1,\alpha}=2\\
...&&&...\\
\maltese \ \ \ K^{a_1+a_2}_\alpha(\alpha-1)&=& -1+ S'_2 \cdot y & c_{a_1+...+a_{2},\alpha}=3\\
...&&&...\\
...&&&...\\
\maltese \ \ \ K^{a_1+...+a_{m-1}}_\alpha(\alpha-1)&=& -1+ S'_{m-1} \cdot y& c_{a_1+...+a_{m-1},\alpha}=3\\
...&&&...\\
K^{a_1+...+a_{m}-1}_\alpha(\alpha-1)&=& -1+ \partial^{a_m-1} S'_{m-1} \cdot y & c_{a_1+...+a_{m}-1,\alpha}=2\\
\maltese \ \ \ K^{a_1+...+a_{m}}_\alpha(\alpha-1)&=&  -y/(y+1)  & c_{a_1+...+a_{m},\alpha}=2\\
\hline
 \alpha&=& S'\cdot y&\\
 K_\alpha(\alpha)&=& \partial^- S'_1 \cdot y& c_{1,\alpha}=-a_{1}-1\\
 K^2_\alpha(\alpha)&=& \partial^{-} S'_2 \cdot y& c_{2,\alpha}=-a_{2}-2\\
 ...&&& ...\\
 K^{m-1}_\alpha(\alpha)&=& \partial^{-} S'_{m-1} \cdot y &c_{m-1,\alpha}=-a_{m-1}-2 \\
 K^{m}_\alpha(\alpha)&=& -y & c_{m,\alpha}=-a_{m}-2\\ 
\hline
\end{array}$$
\caption{Orbits of $\alpha$ and $\alpha-1$, for $\alpha \in (\alpha^-, r]$.}
\label{eq:trajB}
\end{table}

Again, one can go from one line to the following
just using the rules \eqref{eq:rules-} for the upper list or \eqref{eq:rules+} for the lower; we just have to check  that
at each stage we actually get values inside the interval $[\alpha-1, \alpha]$.

As far as the orbit of $\alpha-1$ is concerned, all items of the list are negative, and we just have to check
that we never drop below $\alpha-1$, therefore the only steps which need some comment
are those corresponding to iterates of $K_\alpha$ of order $a_1+...+a_k$
(marked by $\maltese$ in the table). Let us observe that by Lemma \ref{L:fareylegacy} (iii) we have 
 $S_k >> S$. Then we have two cases: either $S'_k >> S'$  and we are done,
or we can write $S' = S_k' Z$, with $Z$ a suffix of $S'$ (hence also a suffix of $'S'={}^t S$), 
and the length of $Z$ is even.
In the latter case, by applying Lemma \ref{L:fareylegacy} to ${}^t S$, one gets $Z >> {}^t S$, hence since 
$y \leq [0; \overline{{}^t S}] \leq [0; \overline{Z}]$ and 
the length of $S_k'$ is odd, we have 
$$y \leq Z \cdot y \Rightarrow S_k' \cdot y \geq S_k' Z \cdot y \Rightarrow  -1 + S_k' \cdot y \geq S' \cdot y$$
proving the required inequality.
The last step is immediate, since $-y/(y+1)>-1/2>\alpha-1$ (note that $\alpha < 1/2$ since $w \in FW_0$).

To check that the values in the lower part of the list of table \ref{eq:trajB} are actually in $[\alpha-1, \alpha]$
we need to prove that $\partial^-S'_k \cdot y  < S'\cdot y$ for  $k \in \{1,...,m-1\}$; indeed, since $\sigma$ is order reversing 
and $'S'={}^tS$, we get
$$ \partial^-S'_k \cdot y < S'\cdot y \iff 
\sigma(\partial^- S'_k \cdot y)> \sigma(S' \cdot y) 
\iff
'(\partial^- S'_k) \cdot y = ({}^t S)_k \cdot y > {}^tS \cdot y
$$ 
and the last inequality holds because we can apply 
Lemma \ref{L:fareylegacy} (iii) to ${}^t S$, yielding
$({}^t S)_k >> {}^tS $.

Finally,  to check that $K^m(\alpha)=-y$ we have to show that $-y>\alpha -1$: indeed, since $y<[0; \overline{{}^tS}]$  we get
${}^tS\cdot y> y$, and $1-\alpha= 'S'\cdot y={}^tS \cdot y>y$.

\smallskip
If we now keep track of the symbolic orbit of $\alpha-1$ and $\alpha$ as we described (see the right column of 
the tables \ref{eq:trajA}, \ref{eq:trajB}), we realize that the values of the coefficients 
$c_{j, \alpha}(\alpha)$ and $c_{j, \alpha}(\alpha-1)$ are
$$
\begin{array}{lll}
(c_{j, \alpha}(\alpha-1))_{1 \leq j \leq m_0} & = & (\underbrace{2,\ldots,2}_{a_1-1},3,\underbrace{2,\ldots,2}_{a_2-1},3,\ldots,\underbrace{2,\ldots,2}_{a_{m-1}-1},3,\underbrace{2,\ldots,2}_{a_m})\\
(c_{j, \alpha}(\alpha))_{1 \leq j \leq m_1} & = & (-a_1-1,-a_2-2,\ldots,-a_m-2)
\end{array}
$$
thus we can compute the matrices ${\bf M}_{\alpha, \alpha-1, m_0}$ and ${\bf M}_{\alpha, \alpha, m_1}$, recovering formula \eqref{eq:mm'}.
\end{proof}

\begin{proof}[Proof of Theorem \ref{T:anatomy}.]
By symmetry (eq. \eqref{eq:ssimmetria}), we need only to check what happens for $\alpha \leq \frac{1}{2}$.
Now, the previous proposition gives us formulas for ${\bf M}$ and ${\bf M}'$, so we just need to check 
that equation \eqref{eq:amatching} holds given these formulas; this is a simple algebraic manipulation as follows.
Let us first prove the case $n = 1$, for which equation \eqref{eq:amatching} becomes
\begin{equation}
 \label{eq:simplematching}
{\bf T}({\bf ST}^2)^N = {\bf S T}^{-N-1} {\bf S T}^{-1} {\bf S}
\end{equation}
(with $N = a_1$). It is well-known and easy to check that $({\bf ST })^3$ is the identity, 
from which it follows 
$${\bf TST } = {\bf S T}^{-1} {\bf S}.$$
Thus, by writing ${\bf TST } = {\bf T}({\bf ST}^2) {\bf T}^{-1}$ and raising both sides to the $N^{th}$ power, we have
$${\bf T} ({\bf ST}^2)^N {\bf T}^{-1} = {\bf S T}^{-N} {\bf S}$$
from which ${\bf T} ({\bf ST}^2)^N = {\bf S T}^{-N} {\bf S} {\bf T}  = {\bf S T}^{-N-1} {\bf T} {\bf S} {\bf T} = 
{\bf S T}^{-N-1} {\bf S} {\bf T}^{-1} {\bf S}$, proving \eqref{eq:simplematching}.
Thus, in general for each $1 \leq k \leq n$ we have the identity 
${\bf T}({\bf ST}^2)^{a_k} = {\bf S T}^{-a_k-1} {\bf S T}^{-1} {\bf S}$, and concatenating all pieces we get 
precisely equation \eqref{eq:amatching}.
\end{proof}

\begin{proof}[Proof of Corollary \ref{C:weakmatching}]
By taking the inverses of both sides of equation \eqref{eq:amatching} and acting on $\alpha$ we get the equality
$$\mathbf{M}_{\alpha, \alpha-1, m_0} \mathbf{T}^{-1} \alpha = \mathbf{S} \mathbf{T}\mathbf{S} \mathbf{M}^{-1}_{\alpha, \alpha, m_1}(\alpha)$$
hence using that $K_\alpha^{m_1}(\alpha) = \mathbf{M}^{-1}_{\alpha, \alpha, m_0}(\alpha)$ and $K_\alpha^{m_0}(x) = \mathbf{M}^{-1}_{\alpha, \alpha-1, m_0}(\alpha-1)$
we get that 
$$-\frac{1}{K^{m_1}_\alpha(\alpha)} + 1 = -\frac{1}{K^{m_0}_\alpha(\alpha-1)}$$
hence for any $k \in \mathbb{Z}$ we have
$$-\frac{1}{K^{m_1}_\alpha(\alpha)} - k \in [\alpha-1, \alpha] \Leftrightarrow
-\frac{1}{K^{m_0}_\alpha(\alpha-1)} - (k+1) \in [\alpha-1, \alpha]$$
thus $c_{m_0+1, \alpha}(\alpha-1) = c_{m_1, \alpha}(\alpha) +1$ and the claim follows.
\end{proof}

Let us conclude this section by studying the ordering between the iterates
of $K_\alpha$, which will be needed in the last section.
As we shall see, this also follows from the combinatorics of the underlying 
Farey words: in particular, the ordering is the same as the ordering between the cyclic permutations
of their Farey structure.



\begin{lemma} \label{L:disjointorbit}
Let $w \in FW_0$: then for any $\alpha \in J_w$ the ordering of the set 
$$\{ K_{\alpha}^j(\alpha-1) \ : \ 0 \leq j \leq m_0 \}$$
of the first $m_0 + 1$ iterates of $\alpha-1$ under $K_\alpha$ is independent 
of $\alpha$. Similarly, the ordering of the set of the first $m_1 +1$ iterates of $\alpha$ is also independent of $\alpha$.
\end{lemma}

\begin{proof}
The first part of the claim can be rephrased as saying that for each $j, j' \in \{0, \dots, m_0\}$ and 
for each $\alpha, \alpha' \in J_w$
$$K_\alpha^i(\alpha-1) < K_\alpha^j(\alpha-1) \textup{ if and only if } K_\alpha^i(\alpha-1) < K_\alpha^j(\alpha-1).$$
Let now $0\leq j<j'\leq m_0$ be fixed; to check that $K^j_\alpha(\alpha-1)$ and $K^{j'}_\alpha(\alpha-1)$ are
ordered in the same way for all $\alpha \in J_w$ it is enough to prove
that they are always different, and to prove this latter statement it
is enough to prove that $K^j_\alpha(\alpha-1) < K^{m_0}_\alpha(\alpha-1)$ for all $\alpha \in J_w$
and all $j \in \{0, \dots, m_0 -1 \}$. This follows from the explicit description of the orbits 
given in the proof of Proposition \ref{R:mconcatenazione}, of which we will keep the notation (see tables \ref{eq:trajA} and \ref{eq:trajB}).
Indeed, in case A the claim holds because of the inequality
$$K_\alpha^j(\alpha-1) < 0 \leq y = K_\alpha^{m_0}(\alpha-1);$$
in case B we have 
$$K_\alpha^{m_0}(\alpha-1) = \frac{-y}{y+1}=-1+ \frac{1}{1+y}=[-1;1,a,...] \ \ \ \ \mbox{ for some } a>1,$$
while the largest of the previous iterates has a continued fraction expansion beginning with $[-1;1,1,...]$.

The corresponding claim about iterates of $\alpha$ is proven in the same way. Indeed, it is enough to check that if $j<m_1$ 
then $K^{m_1}_\alpha(\alpha) < K^j_\alpha(\alpha)$: this is obvious in case B, while in case A we just have to check that 
$\partial^-y<\partial^-S_j \cdot y$; now, since $y < S \cdot y$ (see also proof of Proposition \ref{R:mconcatenazione}, case A)
and $S << S_j$ (Lemma \ref{L:fareylegacy}-(iii)), one gets 
$\partial^-y<\partial^-S \cdot y \leq \partial^-S_j \cdot y$ as claimed.
\end{proof}

\begin{proposition}\label{P:order}
Let $w \in FW_0$ a Farey word, and $w=w'w''$ be its standard factorization, and denote $j_0:=|w'|_0$ and $j_1=|w''|_1$.
Then for each $\alpha \in J_w$ the following holds:
$$
\begin{array}{c}
K_\alpha^{j_0}(\alpha-1)=\min\{K_\alpha^{j}(\alpha-1) \ :1\leq j\leq m_0\} \\
K_\alpha^{j_1}(\alpha)=\max\{K_\alpha^{j}(\alpha) \ :1\leq j\leq m_1\}. 
\end{array}
$$
\end{proposition}

\begin{proof}
By Lemma \ref{L:disjointorbit} and since all the maps $\alpha \to K_\alpha^j(\alpha-1)$ for $0 \leq j \leq m_0$
are continuous on the closure of $J_w$ (they are given by equation \eqref{eq:mm'}), it is sufficient to verify the statement 
for $\alpha = \alpha^+$ the right endpoint of $J_w$. 
In this case, by Proposition \ref{R:mconcatenazione}, the first iterates of $\alpha-1$ are given by 
$$K_\alpha^{a_1 + \dots + a_k + h}(\alpha-1) = -1 + \partial^h S_k \cdot \alpha^+$$
with $0 \leq h \leq a_{k+1} - 1$, $0 \leq k \leq m-1$.
Note that the homeomorphism $\phi$ defined in section \ref{S:qumtervals} semiconjugates the shift 
in the binary expansion with the map $\partial$, hence 
for each prefix $v$ of $w$ we have 
\begin{equation} \label{eq:phitoK}
\phi(0.\overline{\tau^l w}) = K_\alpha^{l_0}(\alpha-1) + 1
\end{equation}
where $l := |v|$ and $l_0 := |v|_0$.
Now, in order to find the smallest nontrivial iterate, recall that by Lemma \ref{cyclic}, the smallest cyclic permutation of $w$ is $w$ itself, 
while the second smallest is $\tau^{|w'|}w$, i.e.
$$w < \tau^{|w'|} w \leq \tau^k w \qquad \textup{for all }1 \leq k \leq m_0 + m_1.$$
Thus, since the homeomorphism $\phi$ is increasing on the interval $[0, 1/2]$, we get by equation \eqref{eq:phitoK}
that for each $j \in \{1, \dots, m_0\}$
$$K_\alpha^{j_0}(\alpha -1 ) \leq K_\alpha^j(\alpha-1)$$
where $j_0 = |w'|_0$ as claimed. 
Similarly, for the orbit of $\alpha$, we know by Proposition \ref{R:mconcatenazione} that the iterates in case $\alpha = \alpha^+$ are given by 
$$K_\alpha^j(\alpha) = \partial^- S_j \cdot \alpha^+ \qquad 1 \leq j \leq m_1.$$
Once again from Lemma \ref{cyclic}, the largest cyclic permutation of $w$ is $\tau^l w$ with $l = |w''|$, hence 
the largest value of $S_j \cdot \alpha^+$ is attained for $j = |w''|_1 = j_1$ and the claim follows.
\end{proof}

\section{Entropy} \label{S:entropy}

We shall now use the combinatorial description of the orbits 
of $K_\alpha$ we have obtained in the previous section 
to derive consequences about the entropy of the maps, proving Theorem \ref{T:main}.

For any fixed  $\alpha \in (0,1)$, the map $K_\alpha$ is a uniformly expanding map 
of the interval and many general facts about its measurable dynamics are known
(see \cite{KU2}).
Indeed, each $K_\alpha$  has a unique absolutely continuous invariant 
probability measure (a.c.i.p. for short) which we will denote $d \mu_\alpha = \rho_\alpha(x) dx$, and the dynamical system  
$(K_\alpha, \mu_\alpha)$ is ergodic. In fact, it is even exact and isomorphic to a Bernoulli shift; moreover, its ergodic properties can also be derived from the properties 
of the geodesic flow on the modular surface.

Let $h(\alpha)$ be the metric entropy of $K_\alpha$ with respect to the measure $\mu_\alpha$: 
we shall be interested in studying the properties of the function $\alpha \mapsto h(\alpha)$. 
Recall that for an expanding map of the interval the entropy can also be given by Rohlin's formula
$$h(\alpha) = \int_{\alpha-1}^\alpha \log |K'_\alpha| \ d\mu_\alpha.$$
Moreover, for maps generating continued fraction algorithms such as $K_\alpha$, 
the entropy is also related to the growth rate of denominators of 
 convergents to a ``typical'' point. 
More precisely, we can define the $\alpha$-convergents to $x$ to be the sequence $(p_{n, \alpha}(x)/q_{n, \alpha}(x))_{n \in \mathbb{N}}$
where 
$$\vect{p_{n, \alpha}(x)}{q_{n, \alpha}(x)} := {\bf  M}_{\alpha, x, n} \cdot \vect{0}{1}.$$ 
For each $x$, the sequence $p_{n, \alpha}(x)/q_{n, \alpha}(x)$ tends to $x$.
Then, for $\mu_\alpha$-almost every $x \in [\alpha-1, \alpha]$ we have 
$$
h(\alpha)=2\lim_{x\to
  +\infty}\frac{1}{n} \log |q_{n,\alpha}(x)|.
$$

As far as the global regularity of the entropy function is concerned, one can easily adapt the strategy of 
\cite{T} to prove that $h(\alpha)$ is H\"older continuous in $\alpha$:

\begin{theorem}\label{TT:tiozzo}
For any $a \in (0, 1/2]$ and any $\eta \in (0,1/2]$, the function $\alpha \mapsto h(\alpha)$ is 
H\"older continuous of exponent $\eta$ on $[a, 1-a]$.
\end{theorem}

On the complement of $\EKU$ we can exploit the rigidity due to the
matching to gain much more regularity. The key tool will be the following proposition.


\begin{proposition}\label{L:NNformula}  
Let $\alpha, \alpha'  \in J_w$ be nearby points which lie both on the same side with respect to the pseudocenter, with $\alpha' < \alpha$. 
Then the following formulas hold:
\begin{eqnarray}
h(\alpha)=[1+(|w|_0-|w|_1) \mu_\alpha([\alpha', \alpha])]h(\alpha')\\
h(\alpha')=[1-(|w|_0-|w|_1) \mu_{\alpha'}([\alpha'-1, \alpha-1])]h(\alpha)
\end{eqnarray}
\end{proposition}

The proof of this proposition follows very closely the proof of the corresponding statement for $\alpha$-continued fractions in 
(\cite{NN}, Theorem 2); a sketch of the argument is included in the appendix.
As a first straightforward consequence of Proposition \ref{L:NNformula} we get the local monotonicity of $h$  on the complement of $\EKU$. 

\begin{corollary}\label{C:localmonotonicity}
The entropy is locally monotone on $[0,1] \setminus \EKU$. More precisely:
\begin{enumerate}
\item the entropy is strictly increasing on $J_w$ if $w\in FW_0 $; 
\item the entropy is constant on $J_w$ if $w=(01)$;
\item the entropy is strictly decreasing on $J_w$ if $w\in FW_1 $.
\end{enumerate}
\end{corollary}


Let us point out that Corollary \ref{C:localmonotonicity} alone is still not enough to 
deduce the monotonicity on $[0,g^2]$ stated in Theorem \ref{T:main}.
Indeed, $h$ is strictly increasing on each open interval $J_w$ for all $w\in FW_0$ and 
the union of all such opens is dense in $[0, g^2]$, 
but to conclude that $h$ is monotone on $[0,g^2]$ we must exclude that $h$ displays pathological
behaviour like the ``devil's staircase'' function. We shall take care of this issue proving 
that $h$ is absolutely continuous.

Let us consider the decomposition $h(t)=h_r(t)+h_s(t)$ where
\begin{equation}\label{eq:hrhs}
h_r(t):=\sum_{w\in FW}   \var_{J_w\cap[0,t]} h , \ \ \ \ h_s(t):=h(t)-h_r(t). 
\end{equation}
Recall the notation $\var_I f$ means the total variation of the function $f$ on the interval $I$.
Intuitively, $h_r$ is the ``regular part'' which takes into account the behaviour of $h$ on the (open and dense) union of the $J_w$, while 
$h_s$ is the remaining ``singular part'', which we will actually prove to be zero.

\begin{lemma}\label{L:loco}
The function $h_s$ is locally constant on $[0,g^2] \setminus \EKU$. 
\end{lemma}
\begin{proof}
If $w\in FW_0$ and $t_1, t_2 \in J_w$, $t_1<t_2$, then the
monotonicity of $h_{|{J_w}}$ implies that
$$h(t_2)-h(t_1)=\var_{J_w\cap [t_1,t_2]} h= h_r(t_2)-h_r(t_1)$$
which implies $h_s(t_2)-h_s(t_1)=0$, whence the claim.
\end{proof}

We shall now need the following lemma in fractal geometry, whose proof we postpone 
to the appendix.

\begin{lemma}\label{L:bdim}
Let $I_1, I_2, \dots$ a countable family of disjoint subintervals 
of some close interval $I \subseteq \mathbb{R}$ and denote
$$\mathcal{G} := I \setminus \bigcup_{i = 1}^\infty I_i.$$
Let the upper box-dimension of $\mathcal{G}$ be $\delta_0$. Given any $\eta$ and $\delta$ with $\eta > \delta > \delta_0$, 
there exists
a constant $C$ such that for any choice of a subsequence $J_1, J_2, \dots$ of the family $\{I_i\}$, 
one gets the following inequality:
$$\sum_{i = 1}^\infty |J_i|^\eta \leq C \left( \sum_{i = 1}^\infty |J_i|\right)^{\eta-\delta}.$$
\end{lemma}

\begin{lemma}\label{L:hsholder}
For every $a$ and $\eta$ in $(0,1/2)$  both functions $h_s$ and $h_r$ are H\"older continuous 
of exponent $\eta$ on the interval $[a,1/2]$.
\end{lemma}
\begin{proof}
By Theorem \ref{TT:tiozzo} 
there is $C=C(a,\eta)$  such that
\begin{equation}\label{eq:hhol}
|h(t)-h(t')|\leq C |t-t'|^{\eta} \ \ \  \forall t,t' \in
[a,1/2].
\end{equation}
Note that in order to prove H\"older continuity of $h_r$ on the whole interval $[a,1/2]$
it is sufficient to show that $h_r$ is H\"older continuous on
$[a,1/2]\cap \EKU$.  
If $\beta, \beta' \in \EKU \cap [a,1/2]$, $\beta<\beta'$, then 
 $$h_r(\beta')-h_r(\beta)=
\sum_{  J_w\subset[\beta, \beta']} \var_{J_w} h .$$ By
 equation \eqref{eq:hhol}, for each $J_w$ we have that  $\var_{J_w} h \leq C  |J_w|^{\eta}$; then we get
\begin{eqnarray*}
|h_r(\beta')-h_r(\beta)|&= 
\sum_{J_w\subset[\beta, \beta']} \var_{J_w} h   
\leq C \sum_{J_w\subset[\beta, \beta']} |J_w|^{\eta} \leq \\
&\leq C C' \left(\sum_{J_w\subset[\beta, \beta']} |J_w|\right)^{\eta-\delta}= C C'|\beta'-\beta|^{\eta-\delta}
\end{eqnarray*}
for any $\delta\in (0,\eta)$; the last line is a direct 
consequence of Lemma \ref{L:bdim} and the fact that 
the box-dimension of $\EKU$ is $0$ (Proposition \ref{Hdim}).
Finally $h_s$ is H\"older continuous since it is a difference of two H\"older continuous functions 
(note the domain has unit length). 
\end{proof}

\begin{lemma} \label{L:extend}
Let $f : I \to I$ be a map of a closed real interval which is H\"older continuous of exponent $\eta > 0$. 
Suppose $E \subseteq I$ is a closed, measurable subset of Hausdorff dimension less than $\eta$, and 
that $f$ is locally constant on the complement on $E$. Then $f$ is constant on all $I$.
\end{lemma}

\begin{proof}
If $f$ is H\"older continuous of exponent $\eta$, it is easy to check 
using the definition of Hausdorff dimension that for any  
measurable subset $E$ of the interval one has the estimate
\begin{equation} \label{eq:dimest}
\textup{H.dim }f(E) \leq \frac{1}{\eta}\ \textup{H.dim }E.
\end{equation}
Now, if $f$ is locally constant on $I \setminus E$, then the image of any 
connected component of the complement of $E$ is a point, hence we have 
$$\textup{H.dim }f(I) = \textup{H.dim }f(E).$$
Now, if $f$ is not constant, then by continuity the image $f(I)$ is a non-degenerate interval, 
hence it has dimension $1$. Thus one has by using equation \eqref{eq:dimest}
$$\textup{H.dim }E \geq \eta \ \textup{H.dim }f(I) = \eta$$
which is a contradiction.
\end{proof}

{\bf Proof of Theorem \ref{T:main}.} 
The second statement
follows from Corollary \ref{C:localmonotonicity}-(2); moreover, by
virtue of the symmetry of $h$, the third statement is a consequence of the first
one, so we just need to prove that the function $h$ is strictly increasing 
on $[0, g^2]$. 

Since the function $h$ is strictly increasing on each $J_w$ which intersects $[0, g^2]$
(Corollary \ref{C:localmonotonicity}-(1)) and the union of the $J_w$ is dense, then the function $h_r$ is strictly increasing 
on $[0, g^2]$. 
The claim then follows if we prove that the function $h_s$ is identically zero, so $h = h_r$.
Now, by Lemma \ref{L:loco} $h_s$ is locally constant on the complement of $\EKU$, it is $\eta$-H\"older 
continuous for any positive $\eta < 1/2$ by Lemma \ref{L:hsholder}, so since $\textup{H.dim }\EKU = 0 < \eta$
the function $h_s$ is globally constant by Lemma \ref{L:extend}.

\qed


\section{Natural extension and regularity properties of the entropy}\label{S:NatExt}

Finally, in this section we shall analyze the regularity properties of the entropy function $h(\alpha)$, 
proving Theorem \ref{T:more}.
In order to do so, we need some results from the theory of $(a,b)$-continued fractions due to 
S. Katok and I. Ugarcovici.
Therefore we will outline here some of the results contained in
\cite{KU1, KU2}; meanwhile, we shall also explain to the reader 
how our constructions relate to the work of Katok and Ugarcovici, 
and translate between the different notations.

The starting point are the ``slow'' maps $f_{a,b} : \mathbb{R} \cup \{\infty\} \to \mathbb{R} \cup \{\infty\}$  defined as
\begin{equation}\label{eq:slow1}
f_{a,b}(y):=\left\{
\begin{array}{ll}
\T y & \textup{if }y<a\\
{\bf S} y & \textup{if }a\leq y<b\\
\T^{-1}y & \textup{if }b\leq y
\end{array}
\right.
\end{equation}
where the parameters $(a,b)$ range in the closed region
\begin{equation}\label{eq:pi}
\mathcal{P}:=\{(a,b)\in \RR^2 \ : \  a\leq 0\leq b, \ b-a\geq 1, \ -ab\leq 1\}
\end{equation}
which is plotted in Figure \ref{fig:abpars}.


\begin{figure}
\includegraphics[scale=0.35]{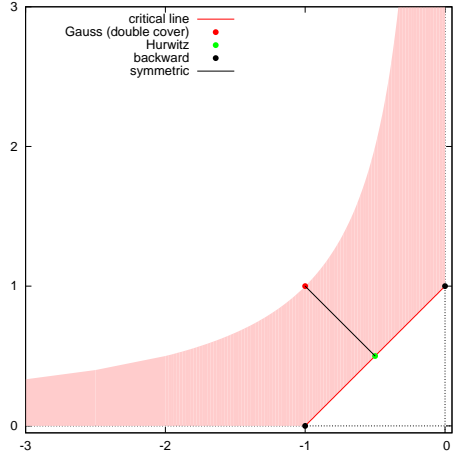} 
\caption{The parameter space of $(a,b)$-continued fractions. We only consider the critical line case $b - a = 1$.}
\label{fig:abpars} 
\end{figure}

An essential role in the theory is played by a condition called {\em
  cycle property}, which we recall briefly. If $f$ is a real map and
$x$ is a point, we call {\em upper orbit} (resp. {\em lower orbit}) of
$x$ the countable set of elements 
$f^k_+(x):= \lim_{t\to x^+} f^k(t)$
(resp. $f^k_-(x):= \lim_{t\to x^-} f^k(t)$), with $k \in \NN$. We say
that the map $f_{a,b}$ satisfies the cycle property at the
discontinuity points $a,b$ if the upper and lower orbit of $a$
eventually collide, and the same is true for $b$.
As a matter of fact for most parameters the cycle property is {\em strong}, meaning that it is the consequence of an identity in 
$SL(2,\ZZ)$, which is stable on an open set. 

To build a geometrical realization of the natural extension one defines the family of maps of the plane 
\begin{equation}\label{eq:slow2}
F_{a,b}(x,y):=\left\{
\begin{array}{ll}
(\T x,\T y) & \textup{if }y<a\\
({\bf S} x,{\bf S} y) & \textup{if }a\leq y<b\\
(\T^{-1}x,\T^{-1}y) & \textup{if }b\leq y.
\end{array}
\right.
\end{equation}
Katok and Ugarcovici prove that each
$F_{a,b}$ has an attractor $D_{a,b}\subset \RR^2$ such that $F_{a,b}$
restricted to $D_{a,b}$ is invertible and it is a geometric
realization of the natural extension of $f_{a,b}$. 
In fact, for most values of $(a, b)$ the attractor $D_{a,b}$ has a simple structure: 
\begin{theorem}[\cite{KU1}] \label{T:KU1}
There exists an uncountable set $\tilde \EE$ of one-dimensional
Lebesgue measure zero that lies on the diagonal boundary $b-a=1$ of $\mathcal{P}$ such that:
\begin{enumerate}
\item for all $(a,b)\in \mathcal{P}\setminus\tilde{\EE}$ the map $F_{a,b}$ has an attractor $D_{a,b}$
(which is disjoint from the diagonal $x=y$) on which $F_{a,b}$ is essentially bijective;
\item the set $D_{a,b}$ consists of two (or one, in the ``degenerate'' case $a b=0$) connected components each having {\em finite rectangular structure}, 
i.e. bounded by non-decreasing step functions with a finite number of steps;
\item almost every point $(x,y)$ off the diagonal $x=y$ is mapped to $D_{a,b}$ after finitely many iterations of $F_{a,b}$.
\end{enumerate}
\end{theorem}


The above result shows that exceptions to the finiteness condition dwell on the 
\emph{critical line} $b-a=1$. For this reason, we only
consider these cases.  With a slight abuse of notation we shall always write
$f_\alpha, F_\alpha, D_\alpha$ rather than $f_{\alpha-1,\alpha},
F_{\alpha-1,\alpha}, D_{\alpha-1,\alpha}$.
Note that if $b-a=1$, the map
$K_b$ is precisely the first return map of $f_{b-1,b}$ on the interval
$[b-1,b)$. 

Note also that in the symmetric case $a+b=0$ (with $b\in [1/2,1]$) the system determined by the 
first return on $[-b,b]$ is equivalent to a twofold cover of the $\alpha$-continued fraction transformation $T_b$.

Let us note that along the critical line 
$b-a=1$ the cycle property is a bit easier to state. Indeed, in this case the upper
orbit of $a$ coincides with the orbit of $a$, while the lower orbit
of $a$ coincides with the orbit of $b$; on the other hand the upper
orbit of $b$ coincides with the orbit of $a$, while the lower orbit of
$b$ is just the orbit of $b$. Thus, in this case the {\em cycle
  property} for $f_{b-1,b}$ is essentially equivalent to the condition
\eqref{eq:matching0} for the fast map $K_b$.
Indeed, the set
$\EKU$ that we explicitely described coincides, up to a countable
set of points\footnote{Precisely, the projection of $\tilde{\mathcal{E}}$ equals $\EKU \setminus \bigcup_{w \in FW^*} \partial J_w$.}, with the projection on the $x$-axis of the set $\tilde
\EE$ mentioned in Theorem \ref{T:KU1}.


Now, for $\alpha \in J_w$ the map $K_\alpha$ satisfies the algebraic
matching condition \eqref{eq:amatching}, hence $f_\alpha$ satisfies
the {\em strong cycle property} (see \cite{KU1}) and by
Theorem \ref{T:KU1} the extension $F_\alpha$ has an attractor
$D_\alpha$ with {\em finite rectangular structure}. We may also
consider the ``first return map'' map 
$\hat{F}_\alpha: \RR\times [\alpha-1,\alpha) \to \RR\times [\alpha-1,\alpha) $
defined as  
$$\hat{F}_\alpha (x,y) :=({\bf T}^{-c_\alpha(y)}{\bf S}x, {\bf T}^{-c_\alpha(y)}{\bf S} y).$$
    Note that the map $K_\alpha$ is a factor of
    $\hat{F}_\alpha$. Obviously the set $D_\alpha \cap \RR\times
    [\alpha-1,\alpha]$ is an attractor for $\hat{F}_\alpha$.  In order
    to compactify the attractor it is convenient to make the change of
    coordinates $\xi={\bf S}x$; then, the natural extension map becomes 
$$\Phi_\alpha (\xi,y) :=({\bf S}{\bf T}^{-c_\alpha(y)}\xi, {\bf T}^{-c_\alpha(y)}{\bf S}y),$$
    which is just the map $\hat{F}_\alpha$ in the new coordinates:
    $\hat{F}_\alpha\circ ({\bf S} \times id)= ({\bf S} \times id)\circ \Phi_\alpha$. The
    map $\Phi_\alpha$ will have the attractor $\Delta_\alpha := {\bf S}D_\alpha =({\bf S} \times
    id)(D_\alpha\cap \RR\times                
    [\alpha-1,\alpha])$, which is bounded and has finite rectangular
    structure.

\begin{figure} 
\centering
\includegraphics[scale=1]{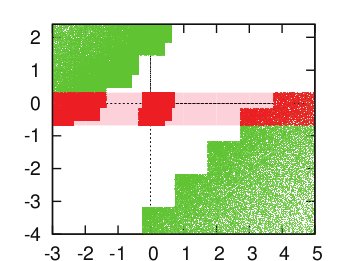}
  \caption{The attractor $D_\alpha$ and its image $\Delta_\alpha$.}
\label{fig:unb-attractor} 
\end{figure}

\subsection{Structure of the attractor $\Delta_\alpha$ and entropy formula.} \label{S:recipe}

In (\cite{KU1}, Section 5) the authors prove that the
attractor $D_\alpha$ has finite rectangular structure providing an
explicit recipe to build it; one can easily translate this recipe in
order to obtain the following analogue description for $\Delta_\alpha$ for $\alpha \in J_w$.
 
Let us fix $w\in FW$ a Farey word, with $m_i:=|w|_i$, and pick $\alpha\in J_w$.
Then the attractor $\Delta_\alpha$ is the union of finitely many rectangles with sides parallel 
to the coordinate axes, and the sides of these rectangles are determined by the dynamics
of $K_\alpha$ prior to the matching, as we now describe. 

\begin{figure}
\centering
\includegraphics[scale=0.5]{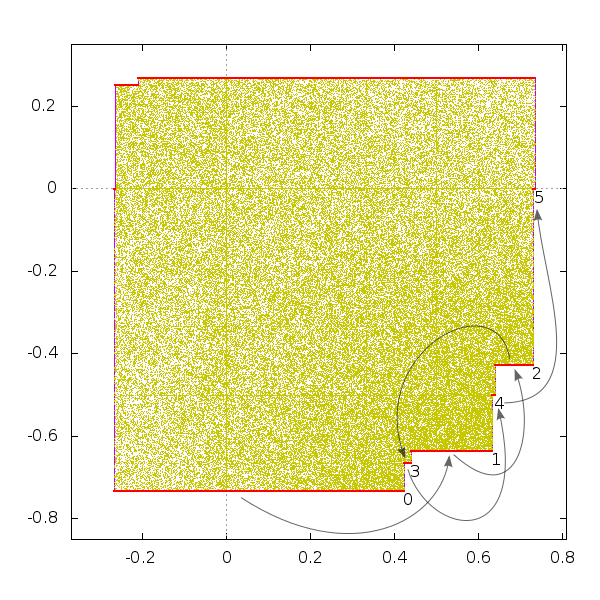} 
\caption{The attractor $\Delta_\alpha$ for $\alpha=4/15$: the numbers and arrows indicate the dynamics of the 
horizontal boundary segments. Note that the vertical ordering of the horizontal segments follows the ordering of the cyclic translates 
of the corresponding Farey word (in this case, $w = 00101$).}
\label{fig:ku4dummies}
\end{figure}

The horizontal segments which delimit $\Delta_\alpha$ are of precisely two types, 
corresponding to the orbits of $\alpha$ and $\alpha-1$, respectively; in particular, 
the set of levels (ordinates) of the horizontal segments on the ``lower-right'' part of $\Delta_\alpha$  
is precisely the set 
$$\{ \alpha-1, K_\alpha(\alpha-1),...,K_\alpha^{m_0}(\alpha-1)\}$$
of iterates of $\alpha-1$ up to the matching, while the set of levels of the horizontal segments on the ``upper-right'' 
side is the set 
$$\{ \alpha, K_\alpha(\alpha),...,K_\alpha^{m_1}(\alpha)\}.$$
The coordinates of the vertical sides of the boundary of $\Delta_\alpha$
(hence the abscissae of its corners) can instead be found in a slightly indirect way, also described in \cite{KU1}:
in order to explain it, let $(x,\alpha)$ and $(y,\alpha-1)$ denote, respectively, 
the upper-right and lower-left corners 
of the attractor $\Delta_\alpha$.



\begin{enumerate}
\item The highest horizontal segment which delimits $\Delta_\alpha$ has 
endpoints $(\frac{y}{y-1},\alpha)$ and $(x,\alpha)$, while the lowest one 
is the segment of endpoints $(y,\alpha-1)$  and $(\frac{x}{x+1}, \alpha-1)$;
\item the horizontal segments which form the upper boundary of
  $\Delta_\alpha$ are images under $\Phi_\alpha$ of the segment of
  endpoints $(\frac{y}{y-1},\alpha)$ and $(x,\alpha)$, and similarly the
  horizontal segments which bound $\Delta_\alpha$ from below are images
  under $\Phi_\alpha$ of the segment of endpoints $(y,\alpha-1)$
  and $(\frac{x}{x+1}, \alpha-1)$;
\item the values of $x$ and $y$ are determined by asking that the
  projection of the horizontal segments bounding $\Delta_\alpha$ from
  above (resp. below) project to adjacent segments; it turns out that it is 
enough to check this condition on a couple of adiacent levels on the top and on the bottom, and this boils down to
  an algebraic relation which only depends on the symbolic orbit of
  $\alpha$ and $\alpha-1$. In particular, it is enough to ask that the right endpoint of the lowest level matches with the 
left endpoint of the level immediately above it. Similarly, one needs to ask that the left endpoint 
of the highest level matches with the right endpoint of the level immediately below it. 
That is, if we let $\pi_1$ be the projection on the $x$-coordinate and $j_0, \; j_1$ be chosen such that
$$K_\alpha^{j_1}(\alpha)=\max\{K_\alpha^{j}(\alpha) \ :1\leq j\leq m_1\}, \ \ \ K_\alpha^{j_0}(\alpha-1)=\min\{K_\alpha^{j}(\alpha-1) \ :1\leq j\leq m_0\}
$$
\end{enumerate}
then, as a consequence of this discussion, the values $x,y$ are determined by the following system:
\begin{equation}\label{eq:qsystem}
\left\{
\begin{array}{l}
{\bf S}{\bf T}{\bf S} y=  
\pi_1(\Phi_\alpha^{j_1}(x,\alpha))\\
{\bf S}{\bf T}^{-1}{\bf S} x= 
\pi_1(\Phi_\alpha^{j_0}(y,\alpha-1))
\end{array}
\right.
\end{equation}
Once $x$ and $y$ are known, then the other vertical levels are obtained by iterating $\Phi_\alpha$ on $x$ and $y$; 
note that, as a consequence, the absicssae of the vertical segments depend \emph{only on the qumterval $J_w$} 
and do not depend on the particular $\alpha$ inside $J_w$. 

We shall now combine this recipe with the results of section \ref{S:matching} and find the following explicit formulas for $x$ and $y$.

\begin{proposition}\label{P:maledetta}
Let $J_w \subset [0,1/2]$ be a \ku and $RL(w)= (a_1,1,...,a_n,1)$; then
$x=[0;\overline{1,a_n,...,1,a_1}]$ and $-y=[0;\overline{a_1,1,...,a_n,1}]$.
\end{proposition}

\begin{proof}
In general if $(a_1,1,...,a_\ell,1)$ and $(a_{\ell+1},1,..., a_n,1)$ is the splitting  of $(a_1,1,...,a_n,1)$ which corresponds to the standard factorization of $w$
then by Proposition \ref{P:order} $j_0=\sum_{i=1}^{\ell} a_i$, $j_1= n-\ell$ and \eqref{eq:qsystem}
 becomes
\begin{equation}\label{eq:qsystem_explicit}
\left\{
\begin{array}{l}
{\bf S}{\bf T}{\bf S} y= 
 {\bf S}{\bf T}^{a_{n-\ell}+2}\ ... \ {\bf S}{\bf T}^{a_2+2} \ {\bf S} {\bf T}^{a_1+1}\ x\\
{\bf S}{\bf T}^{-1}{\bf S} x= {\bf S}{\bf T}^{-3}\ ({\bf S}{\bf T}^{-2})^{a_\ell-1} \   ...
 \ {\bf S}{\bf T}^{-3}  \ ({\bf S}{\bf T}^{-2})^{a_1-1} \ y
\end{array}
\right.
\end{equation}

Let us point out that, by Lemma \ref{L:fareylegacy}-(ii),
 $(a_1-1,a_2,...,a_n)$ is palindrome,
thus $(a_{n-\ell},
a_{n-\ell-1},...,a_2,a_1-1)=(a_{\ell+1},...,a_{n})$; on the other hand, since
$(a_{\ell+1},...,a_{n})$ has Farey structure as well,  Lemma
\ref{L:fareylegacy} implies that
$$(a_{n-\ell},                                                                            
a_{n-\ell-1},...,a_2,a_1-1)=
(a_{\ell+1},...,a_{n})= (a_{n}+1, a_{n-1}, ... , a_{\ell+2}, a_{\ell+1}-1);$$
therefore
 the
first equation of \eqref{eq:qsystem_explicit} can be written as
\begin{equation}\label{eq:palin}
{\bf S}{\bf T}{\bf S} y= {\bf S}{\bf T}^{a_{n}+3} {\bf S}{\bf T}^{a_{n-1}+2}\ ... 
\ {\bf S} {\bf T}^{a_{\ell+2}+2}\ {\bf S} {\bf T}^{a_{\ell+1}+1} x.
\end{equation}
Note that by applying the equality ${\bf TST} = {\bf S}{\bf T}^{-1} {\bf S}$ one gets for each $k$
$${\bf T}^{k+1}{\bf S}{\bf T} = {\bf T}^k {\bf TST } = {\bf T}^k {\bf ST}^{-1} {\bf S}$$
hence by applying this identity to each block on the right-hand side of \eqref{eq:palin} we get 
\begin{equation}\label{eq:palin2}
y = {\bf ST}^{a_n +1 }{\bf ST}^{-1} {\bf ST}^{a_{n-1}} {\bf ST}^{-1} \dots {\bf ST}^{-1} {\bf ST}^{a_{\ell+1}} x.
\end{equation}
Similarly, in order to modify the second equation of \eqref{eq:qsystem_explicit}, we note that 
by leveraging the elementary identity ${\bf T}^{-1} {\bf S}{\bf T}^{-1}={\bf STS}$ we get for each $k$ 
the equality 
$${\bf T}^{-1}({\bf ST}^{-2})^k = ({\bf T}^{-1}{\bf ST}^{-1})^k {\bf T}^{-1} = ({\bf STS})^k {\bf T}^{-1} = {\bf S T}^k{\bf ST}^{-1}.$$
Now, if we apply it to each block on the right-hand side of the second line of \eqref{eq:qsystem_explicit}, we get the equation 
\begin{equation} \label{eq:palin3}
x = {\bf ST}^{-1} {\bf ST}^{a_{\ell}}{\bf ST}^{-1} \dots {\bf ST}^{-1} {\bf ST}^{a_1-1} y.
\end{equation}

We will just prove the claim for $y$, the other case following in the same way.
By putting together \eqref{eq:palin2} and \eqref{eq:palin3}, one finds that 
 $y$ satisfies the fixed point equation $y=Gy$ with


%
%
%
%
%
%

$$G = {\bf S}{\bf T}^{a_{n}+1} {\bf S} {\bf T}^{-1} {\bf S} {\bf T}^{a_{n-1}}{\bf ST}^{-1}
\dots 
{\bf S}{\bf T}^{a_2} {\bf S} {\bf T}^{-1} {\bf S} {\bf T}^{a_1 -1}{\bf ST}^{-1}
$$


Again, since $(a_1,...,a_n)$ has Farey structure, we can use  Lemma \ref{L:fareylegacy}-(iii)
to infer that $(a_{n}+1,a_{n-1},...,a_2, a_1-1)=(a_1,...,a_n)$, hence $G$ can be expressed as
$$G = {\bf S}{\bf T}^{a_{1}}
\ {\bf S T}^{-1} 
\ {\bf S}{\bf T}^{a_{2}} \
 \ {\bf S T}^{-1}
\ ...
 \ {\bf S}{\bf T}^{a_{n-1}} \
 \ {\bf S}{\bf T}^{-1}
\ {\bf S}{\bf T}^{a_n} 
 \ {\bf S}{\bf T}^{-1}
$$
Recalling properties \eqref{eq:evenl} we can check  that setting
$$\check{G}:=  {\bf S}{\bf T}^{-a_{1}}
\ {\bf S T} 
\ {\bf S}{\bf T}^{-a_{2}} \
\ {\bf S T}
\ ...
 \ {\bf S}{\bf T}^{-a_{n-1}} \
 \ {\bf S}{\bf T}
\ {\bf S}{\bf T}^{-a_n} 
 \ {\bf S}{\bf T}
$$
one gets that $-y=\check{G}(-y)$; on the other hand it is immediate to check that $\check{G}$ coincides with the string action induced by $Z=(a_1,1,...,a_n,1)$, and since $-y>0$ we get
$-y=[0;\overline{a_1,1,...,a_n,1}]$.
\end{proof}

For instance, in the case $\alpha=\frac{1}{N+1}$ we get $j_0=N, \ j_1=1$,
so \eqref{eq:qsystem} reads
$${\bf S}{\bf T}{\bf S} y= {\bf S}{\bf T}^{N+1}x, \ \ \ \ {\bf S}{\bf T}^{-1}{\bf S}x=({\bf S}{\bf T}^{-2})^Ny, $$ and a
simple computation yields $x=[0;\overline{1,N}]$,
$-y=[0;\overline{N,1}]$.


The map $\Phi_\alpha$ admits the invariant density $
(1+xy)^{-2} dx dy$; and it is then easy to check that  the 
a.c.i.p. for $K_\alpha$ is $d\mu_\alpha = \rho_\alpha(t)dt$ with invariant density 
\begin{equation}\label{eq:rho}
\rho_\alpha(t):=\left( \int_{\Delta_\alpha\cap\{y=t\}} (1+xy)^{-2} dx \right) /  \left(\int_{\Delta_\alpha} (1+xy)^{-2} dx\ dy \right).
\end{equation}

Moreover, for each $\alpha$ the following formula holds (see \cite{KU2}):
\begin{equation}\label{eq:KU2}
h(\alpha) \int_{\Delta_\alpha} \frac{dx\ dy}{(1+xy)^{2}}  =\frac{\pi^2}{3}.
\end{equation}

\subsection{Consequences} Formula \eqref{eq:KU2} says that, instead of studying the behaviour of the entropy, we may 
just study the function  
$$\alpha \mapsto A_\alpha:=\int_{\Delta_\alpha} \frac{dx\ dy}{(1+xy)^{2}} ,$$
and the explicit description of $\Delta_\alpha$ provides us with an effective tool to do it.

\medskip
\textbf{Proof of Theorem \ref{T:more}.}
The function $\alpha \mapsto
A_\alpha$ is smooth on $J_w$, since the levels of the vertical segments which bound $\Delta_\alpha$ 
are the same for all $\alpha \in J_w$, while the levels of the horizontal segments vary analytically with $\alpha$.
Thus, by equation \eqref{eq:KU2} the function $\alpha \mapsto h(\alpha)$ is smooth as well on each qumterval.

In order to prove the second claim, let us prove that the invariant densities $\rho_\alpha$
are locally bounded from below. In order to do so, let $\alpha \notin \EKU$, and 
$J_w = (\alpha^-, \alpha^+)$ be the qumterval to which $\alpha$ belongs. 
Now, by formula \eqref{eq:leftendpoint} and Proposition \ref{P:maledetta} we have that 
$$1 - \alpha^- = [0; \overline{{}^t S}] = [0; \overline{1, a_n, \dots, 1, a_1}] = x;$$
on the other hand, recall that by the discussion in section \ref{S:recipe} the right endpoint of the lowest
horizontal boundary in $\Delta_\alpha$ has abscissa 
$$x_0 = \frac{x}{x+1} = \frac{1-\alpha^-}{2-\alpha^-} \geq \frac{1}{3}$$ 
since $\alpha^- \leq \alpha \in [0,1/2]$, therefore the following inclusion holds:
\begin{equation}\label{eq:vrectangle}
\Delta_\alpha \supset [0,1/3] \times[\alpha-1,\alpha].
\end{equation}
As a consequence, we can bound the invariant density $\rho_\alpha$ by writing for each $t \in [\alpha-1, \alpha]$
$$\int_{\Delta_\alpha\cap\{y=t\}} (1+xy)^{-2} dx \geq \int_0^{1/3} (1+xt)^{-2} dx \geq \frac{1}{3} \cdot 2^{-2}$$ 
from which, using \eqref{eq:rho} and \eqref{eq:KU2}, it immediately follows that $\rho_\alpha(t) \geq
\frac{1}{12 A_\alpha} = \frac{h(\alpha)}{4 \pi^2}$.  
Now, by Proposition \ref{L:NNformula} the difference quotient of the entropy 
function $h(\alpha)$ on qumtervals is given in terms of $\rho_\alpha$
and the difference $|w|_0 - |w|_1$: 
$$\frac{h(\alpha)-h(\alpha')}{\alpha-\alpha'} = (|w|_0 - |w|_1) h(\alpha') \frac{1}{\alpha- \alpha'} \int_{\alpha'}^\alpha \rho_\alpha. $$
Thus, by combining it with the previous lower bound we get for each $\alpha \in J_w$
\begin{equation}\label{E:deriv}
|h'(\alpha)| \geq C_\alpha ||w|_0 - |w|_1|
\end{equation}
where $C_\alpha$ is bounded away from zero as long as $\alpha$ is bounded away from $0$ or $1$.
Now, let us pick $\alpha \in \EKU$, $\alpha \neq 0, 1$: for any $\alpha'$ sufficiently close to $\alpha$ we have
$$|h(\alpha')-h(\alpha)| = \sum_{J_w \subseteq [\alpha, \alpha']} \var_{J_w} h 
\geq \sum_{J_w \subseteq [\alpha, \alpha']} C ||w|_0 - |w|_1| |J_w| $$
hence, since $\EKU$ has measure zero,  
$$|h(\alpha')-h(\alpha)| \geq C \inf_{J_w \subseteq [\alpha, \alpha']} ||w|_0 - |w|_1|
 |\alpha -\alpha'|.$$
Let us first assume $\alpha \neq g, g^2$. Then by Lemma \ref{L:difference} the difference $||w|_0 - |w|_1|$ 
tends to $\infty$ as soon as $\alpha'$ 
tends to some $\alpha \in \EKU$, thus $h$ is not differentiable (and not even Lipschitz continuous) at $\alpha$.

Suppose instead $\alpha = g^2$ (the other case is analogous by symmetry). Then we know $h$ is constant to the 
right of $\alpha$; on the other hand, by equation \eqref{E:deriv} and the fact that 
$||w|_0 - |w|_1| \geq 1$ to the left of $\alpha$, we get by the same reasoning as before that 
$$\liminf_{\alpha' \to \alpha^-} \frac{h(\alpha)-h(\alpha')}{\alpha-\alpha'} \geq C >0$$
hence the function $h$ is not differentiable at $\alpha$. Finally, since $g^2$ is an accumulation 
point of parameters in $\EKU$ for which the derivative is unbounded, then $h$ is also not locally 
Lipschitz at $\alpha = g^2$.

\qed

In the same way, one can also use formula \eqref{eq:KU2} to prove the following 
asymptotic estimate, which is analogous to the result obtained in \cite{NN}
for the family of Nakada's $\alpha$-continued fractions.

\begin{proposition}\label{P:limit}
The asymptotic behaviour of $h$ at $0$ is 
\begin{equation}\label{eq:asympt}
h(t)\sim \frac{\pi^2}{3\log (1/t)} \ \ \ \ \ \ \mbox{as} \ \ \ t\to 0^+,
\end{equation}
 hence
$\lim_{t\to 0^+} h(t)=0$, and $h$ is not locally H\"older continuous at $0$.
\end{proposition}

\begin{figure}
\centering
\includegraphics[scale=0.4]{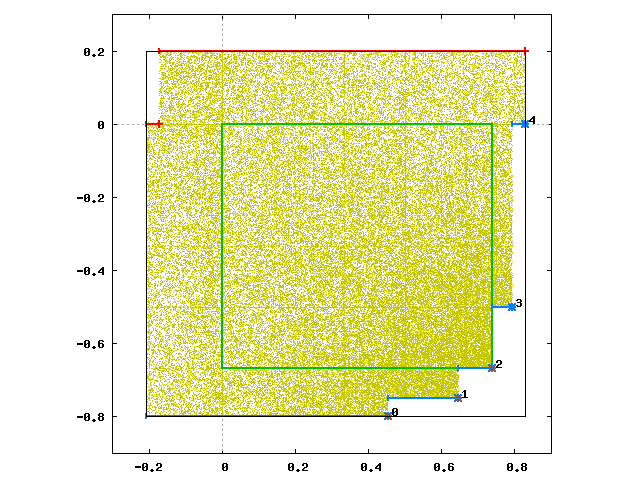}  
\caption{The attractor $\Delta_\alpha$ for $\alpha=1/5$: the inner and outer rectangles give the lower and 
upper bounds for the entropy as in the proof of Propositon \ref{P:limit}.}
\label{fig:attractor-b}
\end{figure}


\begin{proof}
We shall use  formula \eqref{eq:KU2} and prove the asymptotic estimate \eqref{eq:asympt} simply checking that
 \begin{equation}\label{eq:aalpha}
A_\alpha = \log(1/\alpha) + O(1) \qquad \textup{as }\alpha \to 0^+.
\end{equation} 
By Theorem \ref{T:main} we know that $\alpha \mapsto A_\alpha$ is
decreasing on $[0,g^2]$, therefore it is enough to prove
\eqref{eq:aalpha} for $\alpha=\frac{1}{N+1}$ with $N$ a positive, even integer.

Following the recipe of \cite{KU1} described earlier in this section we see that,
for $\alpha=\frac{1}{N+1}$, the attractor $\Delta_\alpha$ has a very simple
structure, which can be completely described.
In particular, it is not difficult to check that the left endpoint of the lowest horizontal boundary
of $\Delta_\alpha$ has coordinates $([0;2,\overline{N,1}], -\frac{N}{N+1})$; the other lower boundaries 
are obtained from the lowest applying the function $\Phi_\alpha$, so that the lower-right corners of the attractor 
are the points $(x_k, y_k) :=\Phi_\alpha^k(x_0,y_0)$ with  
$$\begin{array}{ll}
x_k= & [0;1,k,\overline{1,N}]\\
y_k=& - \frac{N-k}{N-k+1}
\end{array}
$$ 
for $1<k<N$. 
Now, if we pick an even value $N=2h$, we get that $-y_h < x_h$, hence 
the attractor contains the square of coordinates 
$[0,-y_h]\times[y_h,0]$.
Integrating the invariant density $(1+xy)^{-2} dx dy$ on this square
we get the lower bound for the measure of the attractor 
$$A_{1/(N+1)} \geq \log (N)-\log(4).$$
On the other hand, for the upper bound we note that, using the notation of section \ref{S:recipe}, 
we have for each $\alpha \notin \EKU$ the inclusion $\Delta_{\alpha} \subseteq [y, x] \times [\alpha-1, \alpha]$:
then by taking $\alpha = 1/N$ and using Proposition \ref{P:maledetta} one gets that 
the attractor $\Delta_{1/N}$  is contained in the rectangle $[-1/(N-1),1-1/(N+2)]\times [1/N-1,1/N]$, 
which leads to the upper bound $A_{1/N} \leq \log(N) + O(1)$ as $N\to +\infty$.
This, together with the previous inequality, proves \eqref{eq:aalpha}. 
\end{proof}

\subsection{Comparison with Nakada's $\alpha$-continued fractions and open questions.}
Let us remark that the study of the entropy $h_N$ in the case of the family $(T_\alpha)$ of $\alpha$-continued 
fractions of Nakada is indeed much more complicated than the case examined in this paper. Actually many statements 
that we proved before should hold also for the family $(T_\alpha)$, but proofs are missing.

The structure of the matching set for $\alpha$-continued fractions is
quite well understood (\cite{CT}, \cite{BCIT}), but in this case matching intervals with
different monotonic behaviours are mixed up in a complicated way (\cite{CT3}), so
even the fact that the entropy $h_N$ attains its maximum value at
$1/2$ is still conjectural.

\begin{figure}
\includegraphics[scale=0.45]{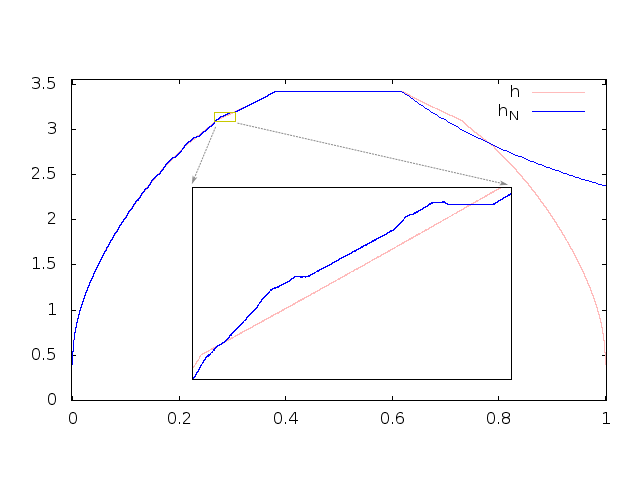} 
\caption{The graph of the entropy $h_N$ of Nakada's $\alpha$-continued fractions (in blue), versus 
the entropy $h$ of continued fractions with $SL(2, \mathbb{Z})$ branches (in pink).}
\label{fig:NaNa} 
\end{figure}

Another feature which is still unproved is the smoothness of entropy on matching
intervals. This is due to the fact that the natural extension has no finite rectangular structure
 when $\alpha$ ranges in a matching interval (see \cite{KSS}).  We conjecture that, as in the case we
examined in this paper, on a matching interval densities are piecewise
continuous, with discontinuity points located on the forward images of
the endpoints (before matching occurs), while the branches of these
densities are fractional transformation which move smoothly with the
parameter (see also \cite{CMPT}, Conj. 5.3).

\section*{Appendix}

We shall now give the proofs of a few technical lemmas we postponed in the main body of the article.

\begin{proof}[Proof of Lemma \ref{Wconcat}]
Let us denote $r = \frac{p}{q}$, $r' = \frac{p'}{q'}$ with $(p, q) = (p', q') = 1$, and $r'' := r \oplus r' = \frac{p+p'}{q+q'}$.
The claim follows immediately from the two facts that for each $0 \leq k \leq q$ we have 
$$\lfloor k r''\rfloor = \lfloor k r\rfloor$$
and for $0 \leq k \leq q'$ we have
$$\lfloor (k+q) r''\rfloor = p + \lfloor k r'\rfloor.$$
To prove the first fact, 
let us note that, using the fact that $r$ and $r'$ form a Farey pair, we have for each $0 \leq k \leq q$
$$0 \leq kr'' - kr \leq \frac{1}{q+q'}$$
thus, from the definition of integer part and the fact that $r$ has denominator $q$, we get
$$\lfloor kr \rfloor \leq k r'' \leq kr + \frac{1}{q+q'} \leq \lfloor kr \rfloor + \frac{q-1}{q} + \frac{1}{q+q'} < \lfloor kr \rfloor + 1$$
hence $\lfloor k r'' \rfloor = \lfloor kr \rfloor$. To prove the second fact, let us note that 
for each $0 \leq k \leq q'$ we have 
$$0 \leq (k+q) r'' - (kr' + p) \leq \frac{1}{q+ q'}$$
thus, similarly as before, we get 
$$\lfloor kr' \rfloor + p \leq (k+q) r'' \leq kr' + p  + \frac{1}{q+q'} \leq \lfloor kr' \rfloor  + \frac{q'-1}{q'} + p + \frac{1}{q+q'} < \lfloor kr' \rfloor + p + 1$$
so $\lfloor (k+q) r'' \rfloor = \lfloor k r' \rfloor + p$.
\end{proof}

\begin{proof}[Proof of Proposition \ref{L:book}]
(1) Let $a^+ = 0.\overline{w}$, where $w$ is a Farey word. Then by Proposition \ref{P:simmetrie} (e), (f) and (h) one has
for each $k \geq 0$
$$w \leq \tau^kw \leq {}^t w < {}^\vee w$$
which, passing from binary words to real numbers, yields $a^+ \in \EB$.
Similarly, by Proposition \ref{P:simmetrie} (e), (f) and (h) we have 
$${}^\vee({}^t w) < w \leq \tau^k w \leq {}^t w$$
which passing to real numbers yields
$$b^- - \frac{1}{2} \leq D^k(b^-) \leq b^- \qquad \forall k \in \mathbb{N}$$
and using that $a^- = b^- - \frac{1}{2}$ yields the claim.

(2) Let $x \in I_w$ with $n := |w|$, and define $a_0 := 0.w$ the dyadic rational with $w$ as its (finite) binary expansion. 
Then the map $D^n$ is an orientation-preserving homeomorphism between the intervals $[a_0, a^+)$ and $[0, a^+)$, hence 
if $x \in [a_0, a^+)$ we have 
$$0 \leq \{2^n x \} < x < a^+$$ 
so $x \notin \EB$. Similarly, $D^n$ maps $(a^-, a_0]$
homeomorphically onto $(0.\overline{{}^t w}, 1]$, hence if $x \in (a^-, a_0]$ then we have 
$\{2^n x\} > x + 1/2 > 0.\overline{{}^t w}$ and $x \notin \EB$.
 
(3) 
Note that every non-degenerate Farey word $w$ is of the form $w = 0v1$ for some (possibly empty) binary word $v$. Then, 
using the fact that ${}^t w = {}^\vee w^\vee$ (Proposition \ref{P:simmetrie} (c)), we have that $a^-$ has a binary expansion 
$a^- = 0.0v0\overline{1v0}$, hence by comparing it with the expansion of $a^+ = 0.\overline{w}$ we get 
$$a^+ - a^- = 0.0\overline{0^N 1} = \frac{1}{2(2^{N+1}-1)}$$
with $N = |v| + 1 = n - 1$, which proves the claim.

(4) Note that by summing up all the lenghts of all the intervals $I_w$ and using (3) we have:
$$\sum_{\stackrel{0 < p < q}{(p, q) = 1}} |I_{W_{p/q}}| = \frac{1}{2} \sum_{\stackrel{0 < p < q}{(p, q) = 1}} \frac{1}{2^q-1} = 
\frac{1}{2} \sum_{\stackrel{0 < p < q}{(p, q) = 1}} \sum_{k = 1}^\infty 2^{-qk} = $$
and by setting  $P := pk$ and $Q := qk$ (note that $k = (P, Q)$) we get 
$$ = \frac{1}{2} \sum_{0 < P<  Q} 2^{-Q} = \frac{1}{2} \sum_{Q = 2}^\infty (Q-1)2^{-Q} = \frac{1}{2}.$$

(5) 
As in the proof of Proposition \ref{Hdim}, it is enough to check that the $\zeta$-function 
$$\sum_{w \in FW} |I_w|^s $$
converges for each $s > 0$. From (3) we have that the above series equals 
$$\sum_{n = 1}^\infty \phi(n) (2^n-1)^{-s} \leq \sum_{n = 1}^\infty n (2^n-1)^{-s} < \infty$$
(where $\phi(n)$ is the Euler function) hence the claim is proven.

\end{proof}
\begin{proof}[Proof of Proposition \ref{P:standard.factorization}.]
Since the claim is true for Farey words of low order, let us prove it by
induction.  Assume the claim is true for all elements of $F_n$ for
$n\geq 1$, and let us consider $w\in F_{n+1}\setminus F_n$ such that
$w=w'w''$ with $w',w'' \in FW$. Without loss of generality we may assume 
that $w \in FW_0$. Now, note that by construction no element of $FW_0$ contains consecutive $1$s,
while all non-degenerate elements of $FW_1$ contain consecutive $1$s: for this reason, 
we have $w'\in FW_0$ and $w''\in FW_0\cup\{01\}.$ Therefore there are
$v',v'' \in FW$ such that $w'=v' \star U_0$ and $w''=v'' \star U_0$,
and so we can write $w=(v'v'')\star U_0$. On the other hand $v=v'v''
\in F_n$, hence by inductive hypothesis $v=v'v''$ is the standard
factorization of $v$, hence $w=w'w''$ is the standard factorization of
$w$.
\end{proof}

\begin{proof}[Proof of Lemma \ref{L:NNformula}.]
 Since the proof follows  closely the same strategy as
in \cite{NN} we give here just a sketch of the main steps.
The idea is based on comparing the return times of $K_\alpha$ 
on the intervals $[\alpha'-1, \alpha-1]$ and $[\alpha', \alpha]$:
by the ergodic theorem, these give information on how the invariant measure
changes. 

Let $J_w$ be the \ku labelled by the Farey word $w$ and let $\alpha, \alpha' \in J_w$ be such that either 
$$ 
\begin{array}{lll}
r\leq \alpha' \leq \alpha< \alpha^+, & S\cdot \alpha'>\alpha \\
\end{array}
$$
or
$$ 
\begin{array}{lll}
\alpha^-<\alpha'<\alpha<r, & {}^t S' \cdot \sigma(\alpha)<\alpha'. 
\end{array}
$$
Then, for every $x\in [\alpha',\alpha]$ there exist two increasing sequences (visiting times) $(n_0(k))_{k\in \NN}$, $(n_1(k))_{k\in \NN}$
such that
\begin{enumerate}
\item $n_0(k)$  and $n_1(k)$ are $k^{th}$-return times on $(\alpha'-1,\alpha-1)$ and $(\alpha',\alpha)$, respectively:
\begin{equation}\label{eq:returns}
\begin{array}{l}
K_{\alpha'}^n(x-1) \in (\alpha'-1,\alpha-1) \iff \ n=n_0(k)  \ \mbox{for some } k \in \NN\\
K_\alpha^n(x) \in (\alpha',\alpha) \iff \ n=n_1(k)  \ \mbox{for some } k \in \NN
\end{array}
\end{equation}
\item although the return times may depend on $x$, their difference just depends on $w$: $n_0(k)-n_1(k)=k(|w|_0-|w|_1)$. 
\item the matching property induces a synchronization of $k^{th}$-returns:
\begin{equation}\label{eq:syncro}
\begin{array}{l}
K_{\alpha'}^{n_0(k)}(x-1)=K_{\alpha}^{n_1(k)} (x)-1\\
{\bf T} {\bf M}_{\alpha',x-1,n_0(k)}={\bf M}_{\alpha,x,n_1(k)}{\bf T} 
\end{array}
\end{equation}
\item Just before the $k$-th return the two orbits are together:
$$K_{\alpha'}^{n_0(k)-1}(x-1)=K_{\alpha}^{n_1(k)-1} (x)-1, {\bf T} {\bf M}_{\alpha',x-1,n_0(k)-1}={\bf M}_{\alpha,x,n_1(k)-1}.$$
\end{enumerate}
It is now possible to choose $x\in (\alpha',\alpha)$ such that both the following conditions hold:
\begin{enumerate}
\item[(a)] $x$ is a typical point for $K_\alpha$, namely
$$
\begin{array}{l}
\lim\limits_{x\to +\infty}\frac{1}{n} \#\{i<n \ : \ K^i_{\alpha}(x) \in
(\alpha',\alpha) \}
=\mu_{\alpha}([\alpha',\alpha]);\\
2\lim\limits_{x\to
  +\infty}\frac{1}{n} \log q_{n,\alpha}(x)=h({\alpha}).
\end{array}
$$
\item[(b)]  $x-1$ is typical for $K_{\alpha'}$, that is:
$$
\begin{array}{l}
\lim\limits_{x\to +\infty}\frac{1}{n} \#\{j<n \ : \ K^j_{\alpha'}(x-1) \in
(\alpha'-1,\alpha-1)
\}=\mu_{\alpha'}([\alpha'-1,\alpha-1]);\\
2 \lim\limits_{x\to
  +\infty}\frac{1}{n} \log q_{n,\alpha'}(x-1)=h({\alpha'})
\end{array}
$$

\end{enumerate}
Therefore, on one hand we have
$$
\lim_{k \to +\infty} \frac{k}{n_0(k)}=\mu_{\alpha'}([\alpha'-1,\alpha-1]), \ \ \ \lim_{k \to +\infty} \frac{k}{n_1(k)}=\mu_{\alpha}([\alpha',\alpha]),
$$
which implies by taking the quotient and using (2) 
$$\lim_{k \to +\infty} \frac{n_{0}(k)}{n_1(k)}=\lim_{k \to +\infty} 1+\frac{k}{n_1(k)}(|w|_0-|w|_1)=1+   (|w|_0-|w|_1) \mu_{\alpha}([\alpha',\alpha]).$$
Then, putting everything together we get
$$
\begin{array}{ccc}
h(\alpha)&=& \lim\limits_{k \to +\infty} \frac{2}{n_1(k)-1} \log q_{n_{1}(k)-1,\alpha}(x)\\[13pt] 
&=& \lim\limits_{k \to +\infty} \frac{n_{0}(k)-1}{n_1(k)-1}\frac{2}{n_0(k)-1}\log q_{n_{0}(k)-1,\alpha'}(x-1)\\[13pt] 
&=& (1+   (|w|_0-|w|_1) \mu_{\alpha}([\alpha',\alpha])) h(\alpha')
\end{array}
$$
which proves the claim.
\end{proof}

\begin{proof}[Proof of Lemma \ref{L:bdim}.]
We may assume, without loss of generality, that $I=[0,1]$ and that the intervals are indexed in 
decreasing size: $|I_1| \geq |I_2| \geq \dots$. 
By definition the upper box-dimension of $\mathcal{G}$ is given by
\begin{equation} \label{eq:dimdef}
\overline{\textup{B.dim }} \mathcal{G} := \limsup_{\epsilon \to 0} \frac{\log N(\epsilon)}{\log (1/\epsilon) }
\end{equation}
where $N(\epsilon)$ is the minimum cardinality of a cover of $\mathcal{G}$ with intervals of diameter less or equal than $\epsilon$. 
Let us now  define 
$$M(\epsilon) := \sup \{i \ : \ |I_i| \geq \epsilon \}.$$
Notice now that any cover of $[0, 1] \setminus \bigcup_{i = 1}^{M(\epsilon)} I_i$ with intervals 
of diameter less than $\epsilon$ necessarily must have cardinality at least $M(\epsilon) + 1$, 
because any such interval intersects at most one connected component. Hence
$$M(\epsilon) \leq N(\epsilon).$$
If we now fix $\delta \in (\delta_0, \eta)$, by \eqref{eq:dimdef} there is some $C > 0$ such that 
$$M(\epsilon) \leq N(\epsilon) \leq C \epsilon^{-\delta} \qquad \forall \epsilon \leq 1.$$
Now
$$\sum_{M(\epsilon/2^k)+1}^{M(\epsilon/2^{k+1})} |I_i|^\eta \leq
\left(\frac{\epsilon}{2^k}\right)^\eta M(\epsilon/2^{k+1}) \leq C
2^{\delta} \epsilon^{\eta-\delta} 2^{k(\delta-\eta)}$$ hence summing over 
$k \geq 0$
$$\sum_{M(\epsilon)+1}^{\infty} |I_i|^\eta  \leq \epsilon^{\eta-\delta} \frac{C 2^{\delta} }{ 1- 2^{\delta-\eta}}.$$
Now, given any subsequence $J_1, J_2, \dots$, 
if we set $\epsilon := \sum_{i =1}^\infty |J_i|$, then all elements in the subsequence have length smaller than $\epsilon$, hence
$$\sum_{i = 1}^\infty |J_i|^\eta \leq \sum_{M(\epsilon)+1}^{\infty}
|I_i|^\eta \leq \left( \sum_{i = 1}^\infty |J_i| \right)^{\eta-\delta}
\frac{C 2^{\delta} }{ 1- 2^{\delta-\eta}}.$$
\end{proof}

\begin{proof}[Proof of Proposition \ref{P:thick}.]
Let us pick $r'\in (0,1)\cap\QQ$, since  $\QQ \cap \EKU=\{0,1\}$ then $r'\notin \EKU$. Therefore there is $w\in FW^*$ such that $r'\in J_w$, and Proposition \ref{P:thick} will be proved once we prove that
\begin{equation}\label{eq:weakmax} 
r'\in J_w \ \ \Rightarrow 
\tilde{J}_{r'}\subset J_w.
\end{equation}
Now, let $S:=RL(w)$ and $r:=[0;S]$ be the pseudocenter of $J_w$ (so that, by virtue of equation
\eqref{eq:qumterval}, $\tilde{J}_r=J_w$).  Let us first assume that
$r'>r$: this means that $r'=[0;ST]$ with $|T|$ even. Since $S'<<S$, it
is clear that the left endpoint $\sigma \beta(\sigma r')$ of $\tilde{J}_{r'}$ 
belongs to $\tilde{J}_r$. To prove that
the right endpoint $\beta(r')$ satisfies $\beta(r') \leq \beta(r)$, let us set $m:=\max\{j
\ : \ |S|^j<|T|\}$. Then,  either (A) $T<<S^{m+1}$ or (B) $T=S^m P $ with
$S=PZ$ (i.e. $T$ is prefix of $S^{m+1}$): We claim that in both cases
$\beta(r') = [0;\overline{ST}] \leq [0;\overline{S}] = \beta(r)$. In case (A) this claim is
trivial, and the same is true in case (B) if $P=S$; on the other hand,
in the case when P is a proper prefix of $S$, by virtue of Lemma
\ref{L:fareylegacy}-(3) one gets that $PZ<ZP$ so that using the Lemma of eq. \eqref{eq:stringlemma}
$$
PPZ<PZP \iff S^{m+1}PS<S S^{m+1} P \iff [0;\overline{ST}]=[0;\overline{S^{m+1} P}]<[0;\overline{S}] 
$$
completing the proof of the inclusion $\tilde{J}_{r'}\subset \tilde{J}_r = J_w$. 

On the other hand, if $r'<r$, then $\sigma(r')>\sigma(r)$ and
$\sigma(r')\in\tilde{J}_{\sigma r}=J_{{}^t\check{w}}$, so the
previous case implies
$$ \tilde{J}_{ r'}=\sigma(\tilde{J}_{\sigma r'})\subset\sigma(\tilde{J}_{\sigma r})= \tilde{J}_{ r}
 $$
and \eqref{eq:weakmax} is thus proven.
%
\end{proof}

\end{document}